\documentclass[12pt]{amsart}
\usepackage{amsmath,amsthm,amsfonts,amssymb,
 graphicx,tikz,fancyhdr,color}%
\usetikzlibrary{matrix,arrows.meta}
\usepackage{tikz-cd}
\usepackage{hyperref,enumitem}
\hypersetup{
    colorlinks,
    citecolor=black,
    filecolor=black,
    linkcolor=black,
    urlcolor=black
}
\DeclareMathOperator{\Tr}{{Tr}} 
\DeclareMathOperator{\Nm}{{N}}
\begin{document} 
\newcommand{\changed}[1]{\textcolor{red}{#1}\message{C}}
\newcommand{\anfangchanged}{\color{red}\message{C}}
\newcommand{\finchanged}{\color{black}}
\newcommand{\tensor}{{\otimes}}
  \newcommand{\cS}{{\mathcal S}}
  \newcommand{\cU}{{\mathcal U}}
  \newcommand{\cF}{{\mathcal F}}
  \newcommand{\cL}{{\mathcal L}}
\newcommand{\Ic}{{\mathcal I}}
\let\paragraph\S
\renewcommand{\S}{{\mathbb S}}
\newcommand{\Octc}{{\Oct_\C}}
\newcommand{\C}{{\mathbb C}}
\renewcommand{\H}{{\mathbb H}}
\newcommand{\N}{{\mathbb N}}
\renewcommand{\O}{{\mathcal O}}
\newcommand{\Z}{{\mathbb Z}}
\renewcommand{\P}{{\mathbb P}}
\newcommand{\R}{{\mathbb R}}
\newcommand{\dimc}{\mathop{dim}_{\C}}
\newcommand{\dimr}{\mathop{dim}_{\R}}
\newcommand{\Lie}{\mathop{Lie}}
\newcommand{\Aut}{\mathop{{\rm Aut}}}
\newcommand{\dom}{D}
\newcommand{\Oct}{{\mathbb O}}
\newcommand{\Gc}{G_{\C}}
\newcommand{\GC}{G_{\C}}
\newcommand{\Hc}{\H_{\C}}
\newcommand{\OC}{\Oct\tensor_\R\C}  
\newcommand{\WC}{{W\tensor_\R\C}}
\newcommand{\X}{\underline{\overline{X}}}
\newcommand{\cdiv}{\mathop{cdiv}}
\newcommand{\SOct}{{\mathbb S}_{\Oct}}
\let\Realpart\Re
\renewcommand{\Re}{\mathop{\Realpart e}}
\let\Impart\Im
\renewcommand{\Im}{\mathop{\Impart m}}
\newcommand{\Offen}{U}
\newtheorem{lemma}{Lemma}[section]
\newtheorem{definition}[lemma]{Definition}
\newtheorem*{claim}{Claim}
\newtheorem{corollary}[lemma]{Corollary}
\newtheorem*{Conjecture}{Conjecture}
\newtheorem*{SpecAss}{Assumptions}
\newtheorem{example}[lemma]{Example}
\newtheorem*{remark}{Remark}
\newtheorem*{observation}{Observation}
\newtheorem*{fact}{Fact}
\newtheorem*{warning}{Warning}
\newtheorem*{remarks}{Remarks}
\newtheorem{proposition}[lemma]{Proposition}
\newtheorem{theorem}[lemma]{Theorem}
\newcommand{\notiz}[1]{{\tt To do:}{\color{red}\tt #1}{\message{XX: #1}}}
\newcommand{\XXX}{XXX}
\numberwithin{equation}{section}
\def\labelenumi{\rm(\roman{enumi})}
\title{%
  Invariants and Automorphisms for Slice Regular Functions:
The Octonionic Case}
\author {Cinzia Bisi \& J\"org Winkelmann}
\begin{abstract}
  We study automorphisms and invariants for the algebra $\Oct$ of octonions
  and octonionic slice regular functions $f:\Oct\to\Oct$.
  \end{abstract}
\address{%
Cinzia Bisi \\
Department of Mathematics and Computer Sciences\\
Ferrara University\\
Via Machiavelli 30\\
44121 Ferrara \\
Italy
}
\email{bsicnz@unife.it \newline
 ORCID : 0000-0002-4973-1053
}
\address{%
J\"org Winkelmann \\
IB 3/111\\
Lehrstuhl Analysis II \\
Fakult\"at f\"ur Mathematik\\
Ruhr-Universit\"at Bochum\\
44780 Bochum \\
Germany\\
}
\email{joerg.winkelmann@rub.de\newline
  ORCID: 0000-0002-1781-5842
}
\thanks{
The two authors were partially supported by GNSAGA of INdAM.
C. Bisi was also partially supported by PRIN \textit{Variet\'a reali e complesse:
geometria, topologia e analisi armonica}. 
}
\maketitle
\tableofcontents
\section{Introduction}

  We are concerned with ``slice regular'' functions on the algebra
  of octonions.
  (See \cite{DM},\cite{CS} for general information on octonions.)
  
The notion of slice regular functions on octonions is a generalization
of the corresponding notion for quaternions.

There is a long history of studying quaternions going back to Hamilton
which has many applications e.g. in geometry and physics.

In view of the theorem of Kervaire and Milnor, which states that $\R$, $\C$,
the quaternion algebra $\H$ and the algebra $\Oct$ of octonions are
the only finite-dimensional real division algebras,
the algebra of octonions $\Oct$
may be regarded as the ``small brother''
\footnote{``small'' in the sense that it has weaker properties.
  In particular $\Oct$ is an algebra which is
  only alternative and not associative.
  Of course $\dim(\Oct)>\dim(\H)$.}
of the quaternionic algebra $\H$,
raising the question up to which degree results on quaternions
remain true in the octonion set up.

This article adresses this issue with regard to ``slice regular fcuntions''.

For the algebra $\H$ of quaternion
numbers the theory of slice regular functions was introduced by G. Gentili and D. Struppa in two seminal papers in 2006 \cite{GS1} and in \cite{GS2}.
They used the fact that $\forall I \in \mathbb{S}_{\mathbb{H}} = \{ J \in \mathbb{H} \,\, | \,\, J^2=-1 \}$ the real subalgebra $\mathbb{C}_{I}$ generated by 1 and $I$ is isomorphic to $\mathbb{C}$ and they decomposed the algebra $\mathbb{H}$ into a "book-structure" via these complex "slices":
$$
\mathbb{H} =\cup_{I \in \mathbb{S}_{\mathbb{H}} }   \mathbb{C}_I
$$
On an open set $\Omega \subset \mathbb{H},$ they defined a differentiable function $f \colon \Omega \to \mathbb{H}$ to be (Cullen or) slice regular if, for each $I \in \mathbb{S}_{\mathbb H},$ the restriction of $f$ to $\Omega_I= \Omega \cap \mathbb{C}_I$
is a holomorphic function from $\Omega_I$ to $\mathbb{H},$ both endowed with the complex structure defined by left multiplication by $I.$ This definition contains all convergent power series of the form:
$$
\sum_{n \in \mathbb{N}_0} w^n a_n 
$$
with $\{ a_n \}_{n \in \mathbb{N}_0} \subset \mathbb{H}.$ \\
On the algebra $\mathbb{O}$ of octonion numbers
the same approach may be used and an analogous book-structure with complex slices holds true as well as the power series expansion in zero with octonionic variable and coefficients, for slice regular functions over $\mathbb{O}$.
\\

There is another, different, but equivalent approach
to slice regular functions introduced by R. Ghiloni
and A. Perotti in 2011, \cite{GP}.
For an alternative $*-$algebra $A$ over $\mathbb{R}$ they
use ``stem functions'' with values in the
complexified algebra $A \otimes_{\mathbb{R}} \mathbb{C},$
denoted by $A_{\mathbb{C}}$. \\
The algebra of octonions $\Oct$ is an alternative $*$-algebra,
so this theory applies to the octonions.
In this article we mostly use this approach.

Let us denote the elements of $\Octc$
as $a+\iota b$ where $a, b \in \Oct$ and $\iota$ is to be considered as the imaginary unit of $\mathbb{C}$ distinguished by the $i$ that appears in
the usual basis for $\Oct$. \\
For any slice regular function
$f:\Oct\to \Oct$, and for any $I \in \S_\Oct$,
(with $\S_\Oct=\{x\in \Oct:x^ 2=-1\}$)
the restriction $f \colon \mathbb{C}_I \to \Oct$
can be lifted through the  map
$\phi_I \colon \Oct_{\mathbb{C}} \to \Oct$,
$\phi_I (a+\iota b) := a+I b$
to a map $\C\cong\C_I\to\Octc$
and it turns out that the lift does not depend on $I.$  In other words, there exists a holomorphic function $F \colon \mathbb{C} \to \Octc$
which makes the following diagram commutative for all $I \in \S_{\Oct}.$ \\
$$
\begin{tikzcd}
 & \mathbb{C} \arrow[r, "F"] \arrow[d, "\phi_I"']
& \Octc \arrow[d, "\phi_I"]  \\
& \Oct  \arrow[r, "f"] & \Oct  
\end{tikzcd}
$$
Conversely if a function $f:\Oct\to\Oct$ admits such a lift, it is
slice regular.

The class of ``slice regular functions'' includes polynomials of the
form $P(w)=\sum_{k=0}^d w^k c_k$ (with $c_k\in\Oct$) and similar
power series $\sum_{k=0}^\infty w^kc_k$ (if convergent).
In particular, using power series development, classical
functions like $\exp$, $\sin$, $\cos$, $\cosh$ extend to
slice regular functions on $\Oct$.

  This notion of ``slice regularity'' easily generalizes to the case where
we consider functions which are defined not globally, but only on a
suitable open subset (``axially symmetric domain'').

For instance, the power series $\sum_{k=1}^{+\infty} w^k\frac{(-1)^{k+1}}{k}$
of the logarithmic function $w\mapsto \log(1+w)$
may be used to define a slice regular function on the unit ball in the
algebra of octonions.

After the first definitions were given, the theory of slice regular functions knew a big development: see, among the others, the following references
\cite{GSSMono}, \cite{BW2}, \cite{BDMW}, \cite{BW1}, \cite{BW3},\cite{BC23} \cite{AB2},\cite{AB1}, \cite{BS}, \cite{BG}, \cite{BM}. \\

The ``essential'' properties of a number or a function should not
be changed by symmetries.

  As a $*$-algebra, the algebra of octonions $\Oct$ admits an 
{\em antiinvolution} $x\mapsto \bar x$ which commutes with all
automorphisms. As a consequence, $\Nm(x)=x\bar x$ and $\Tr(x)=x+\bar x$ are
{\em invariant} under automorphisms. In fact, 
we have the equivalence (for $z,w\in\Oct$):
\[
\Nm(z)=\Nm(w)\text{ and } \Tr(z)=\Tr(w)
\iff \exists \,\,\, \phi\in \Aut(\Oct):\phi(z)=w
\]
where $\phi$ is an automorphism of $\Oct$ as an $\R$-algebra.
(See \cite{DentoniSce}, $L_4$, p.260.)

This raises the question whether a similar correspondence holds
not only for the elements in the algebra, but also for slice-regular
functions of this algebra.

As it turns out, essentially this is true, but only via the associated
stem functions and up to a condition on the multiplicity with which
  values in the center of $\Octc$ are assumed (see \paragraph\ref{ss-cdiv}).
  To state the latter condition, in \paragraph\ref{ss-cdiv} 
  we use the notion of a
  ``{\em central divisor}'' $\cdiv$
  which we introduced in \cite{BWH}.
  We also assume that $f,h$ are not slice preserving, i.e., the image
  of their stem functions is not contained in the center of $\Octc$.
  Using this, we prove (Theorem~\ref{mainth}) that, given two slice
  regular functions $f$ and $h$ with stem function $F$ and $H$,
  they have the same invariants ($\Nm,\Tr,\cdiv$) if and only if
  there is a holomorphic map $\phi$ with values in $\Aut(\Octc)$
  such that
  \[
  \forall z: F(z)=\phi(z)\left(H(z)\right)
  \]

  The complete statement is in \paragraph 2, in \paragraph 3 we discuss
  notions like $\Nm$, $\Tr$ and do some preparations, in \paragraph 4 we lay out the
  strategy of the proof of the main theorem. The remainder of the paper
  consists of the actual proofs.

  \begin{remark}
  Here we discuss octonionic slice regular functions.  
In an earlier paper (\cite{BWH}) we obtained similar results for
 the algebra of slice regular functions with values in the algebra of quaternions $ \mathbb{H}$
 or the Clifford algebra $ \mathbb{R}_3 =\mathbb{H} \oplus \mathbb{H}$.

 We would like to emphasize that, while the results are similar, for the proofs
 we need
 quite different methods in the two cases (octonions versus
 quaternions (and $\R_3$)).
\end{remark}
    \section{Main theorem}

  Our Main Theorem is the following :
\begin{theorem}\label{mainth}
  Let $\Oct$ be the algebra of octonions,
  $\Octc=\Oct\tensor_\R\C$ its complexification and
    $\GC=Aut(\Octc)$ the group of
    $\C$-algebra automorphisms of $\Octc$. Let $\dom\subset\C$ be a symmetric domain
  and let $\Omega_D\subset \Oct$ denote the corresponding axially symmetric
  domain (as defined in Definition~\ref{def-sym}).

  Let $f,h:\Omega_D\to \Oct$ be slice regular functions
  and let $F,H:\dom\to\Octc$
  denote the corresponding stem functions.
  \begin{enumerate}[label=\alph*)]
  \item
  Assume that neither $f$ nor $h$ is slice preserving.

   Then the following are equivalent:
  \begin{enumerate}[label=(\roman*)]
  \item
    $f$ and $h$ have the same invariants $\cdiv$, $\Tr$, $\Nm$.
  \item
    $F$ and $H$ have the same invariants $\cdiv$, $\Tr$, $\Nm$.
  \item
    $\cdiv(F)=\cdiv(H)$ and for every $z\in\dom$ there exists an element
    $\alpha\in Aut(\Octc)=\GC$ such that $F(z)=\alpha(H(z))$.
\item
        There is a holomorphic map $\phi:\dom\to\GC$
    such that $F(z)=\phi(z)\left(H(z)\right)\ \forall z\in\dom$.
  \end{enumerate}
  \item
  Assume that $f$ is slice-preserving.
    Then the following are equivalent:
  \begin{enumerate}[label=(\roman*)]
  \item
    $f=h$.
  \item
    $F=H$.
  \item
    For every $z\in\dom$ there exists an element
    $\alpha\in Aut(\Octc)=\GC$ such that $F(z)=\alpha(H(z))$.
\item
    There is a holomorphic map $\phi:\dom\to\GC$
    such that $F(z)=\phi(z)\left(H(z)\right)\ \forall z\in\dom$.
  \end{enumerate}
    \end{enumerate}
\end{theorem}
  
\begin{remark}
  The notion of a ``central divisor'' is defined only if the function
  is not slice preserving. This is similar to the
  classical situation in complex analysis
  where the divisor of a holomorphic function is defined only if it is not
  constantly zero.
\end{remark}

The following example illustrates the need for a special treatment
of slice preserving functions. Namely, we show that if one of the
functions is slice preserving and the other one not, then they may share
the same
invariants
$\Nm$, $\Tr$ without being related by a map to the automorphism
group.

\begin{example}
  Let $D=\C\setminus\R$.
  Let $I,J\in\Oct$ with $I^2=-1=J^2$ and $IJ=-JI$.
  Define $H\equiv 0$ and $F:D\to\Octc$ as
  \[
  F(z)=\begin{cases}
    I\tensor z +J\tensor\iota z & \text{ if $\Im(z)>0$}\\
    I\tensor z -J\tensor\iota z & \text{ if $\Im(z)<0$}\\
  \end{cases}
  \]
  Then $F$ is a stem function with $\Nm(F)=0=\Tr(F)$, but evidently
  there is no holomorphic map $\phi:D\to\Aut(\Octc)$ with
  \[
  F(z)=\phi(z)\left(H(z)\right)=\phi(z)\left( 0 \right)
  \ \forall z\in\dom,
  \]
  since every automorphism of the algebra $\Octc$
  fixes the zero element $0$.
\end{example}

The main theorem (Theorem~\ref{mainth}) is proved in \paragraph\ref{pf-main}.
Details concerning conjugation, trace and norm are discussed
  in \paragraph\ref{sect-prep}.
For the definition of the ``central divisor'' $\cdiv$, see \paragraph
\ref{ss-cdiv}.

\section{Preparations}\label{sect-prep}

Here we collect basic facts and notions needed for our main result.
First we discuss conjugation, norm and trace,
then types of domains, then slice regular
functions and stem functions, followed by investigating conjugation,
norm and trace for function algebras.

We will formulate these preparations for arbitrary alternative algebras
even if we will apply them only to the case of octonions.

\subsection{Conjugation, norm and trace}

Let $A$ be an
alternative
$\R$-algebra  with $1$
and let $x\mapsto \bar x$ be an {\em
  antiinvolution}, i.e., an $\R$-linear map such that
$\overline{xy}=(\bar y)\cdot(\bar x)$
and $\overline{(\bar x)}=x$ for all $x,y\in A$.
(An $\R$-algebra with an antiinvolution is often called
$*$-algebra.)

\begin{definition}\label{def-n-tr}
  Given an $\R$-algebra $A$ with antiinvolution $x\mapsto\bar x$,
  we define:
  \begin{align*}
    \text{Trace: } & \Tr(x)=x+\bar x\\
    \text{Norm: } & \Nm(x)=x\bar x\\
  \end{align*}
\end{definition}

Consider
\[
C=\{x\in A: x= \bar x\}.
\]

We assume that $C$ is {\em central} and associates with
all other elements, i.e.,
\[
\forall c\in C, x,y\in A:  cx=xc\text{ and }c(xy)=(cx)y
\]
It is easy to verify that $C$ is a subalgebra
(under these assumptions, i.e., if $C$ is assumed to be central).

\begin{lemma}
  Under the above assumptions the following properties hold:
  
  \begin{enumerate}
    \item
      $\forall x\in \R:x=\bar x$.
        \item
    $\forall x\in A: \Nm(x),\Tr(x)\in C=\{ y\in A: y=\bar y\}$
  \item
    $\forall x\in A: x\bar x=\bar x x$
  \item
    $\forall x\in A: \Nm(x)=\Nm(\bar x)$.
  \item
    $\forall x,y\in A: \Nm(xy)=\Nm(x)\Nm(y)$.
  \end{enumerate}
\end{lemma}
\begin{proof}
See \cite{BWH}, Lemma~2.2.
\end{proof}

{\em Other notions.}

If $\Tr(x)\in\R\ \forall x\in A$,
then
$\frac12\Tr(x)$ is often called {\em real part} of $x$, sometimes
denoted as $x_0$.

$N(x)$ is also called the {\em symmetrization} of $x$ and denoted as
$x^s$.

\subsection{(Axially) symmetric domains}\label{sect-dom}
\begin{lemma}\label{lem-1-2}
  Let $\Omega$ be an open subset of
  the algebra $\Oct$ of octonions.
  Let $\S_{\Oct}=\{q\in\Oct:q^2=-1\}$.
  
  Then the following are equivalent:
  \begin{enumerate}
  \item
    There exists an open subset $D\subset \C$ such that
    \[
    \forall x,y\in\R:\forall J\in\S_{\Oct}:\ 
    x+yi\in D\quad \iff\quad x+yJ\in\Omega
    \]
  \item
    \[
    \forall J,K\in\S_{\Oct}:\forall x,y\in\R:
    x+yJ\in\Omega\ \iff\ x+yK\in\Omega
    \]
  \item
    There is a subset $M\subset\R\times\R^+_0$ such that
    \[
    \Omega=\{q\in\Oct:(\Tr(q),\Nm(q))\in M\}
    \]
  \item
    $\Omega$ is invariant under the action of $\Aut(\Oct)$.
  \item
    $\Omega$ is invariant under the action of $O(W)$, the group
    of orthogonal transformations of $W=\{q\in\Oct:\Tr(q)=0\}$
    acting naturally on $W$ and acting trivially on $\R$.
  \end{enumerate}
\end{lemma}

\begin{proof}
  $(i)\iff(ii)$ is obvious.

  $(ii)\iff(iii)$:
  Let $q=x+yJ\in\Oct$ with $x,y\in\R$, $J\in\SOct$.
  Then $2x=\Tr(q)$ and $x^2+y^2=\Nm(q)$. Hence for any given
  $(t,n)\in\R\times\R^+_0$ we have
  \begin{align*}
  &\{q=x+yJ\in\Oct: \Nm(q)=n, \Tr(q)=t\}\\
  =&\left\{\frac t2+yJ: J\in\SOct,
  t\in\R, y\in\R^+_0,
  y^2=n-t^2/4\right\}.\\
  \end{align*}
  This yields $(ii)\iff(iii)$.

  $(iii)\iff(iv)$: See Corollary~\ref{n-tr-auto}.

  $(iv)\iff(v)$:
    Follows from Corollary~\ref{n-tr-auto} in combination with
    Proposition~\ref{auto-ortho}.
    
\end{proof}

  \begin{definition}\label{def-sym}
    \begin{enumerate}
    \item
      A domain $D$ in $\C$ is called {\em symmetric} if
      \[
      z\in D\iff \bar z\in D.
      \]
    \item
      A domain $\Omega$ in $\Oct$ is called {\em axially
        symmetric} if it satisfies one (hence all) of the
      properties of Lemma~\ref{lem-1-2}.
    \end{enumerate}
  \end{definition}

  In the situation of Lemma~\ref{lem-1-2} $(i)$ we write
  $\Omega_D=\Omega$, since $D$ and $\Omega$ are in one-to-one-correspondence.
  
\subsection{Slice and Stem functions}

\begin{definition}
  Let $D$ be a symmetric
  domain in $\C$ and let $\Omega=\Omega_D$
  be the associated {\em axially symmetric domain}, i.e.,
  \[
  \Omega=\Omega_D=\{x+yJ:x,y\in\R,J\in\SOct, x+yi\in D\}
  \]
  \begin{enumerate}
    \item
    A function $F:D\to\Octc$ is a {``\em stem function''} if
    \[
    \forall z\in D:\ \overline{F(\bar z)}=F(z)
    \]
    where we conjugate the complex
    part of the tensor product $\Octc=\Oct\tensor_\R\C$.
  \item
    A function $f:\Omega_D\to\Oct$ is a slice function,
    if there exists a stem function $F=F_1+\iota F_2$ ($F_i:D\to\Oct$) such that
    \[
    \forall x+yi\in D,J\in\SOct:
    f(x+yJ)=F_1(x+yi)+J F_2(x+yi)
    \]
  \item
    A function $f:\Omega_D\to\Oct$ is a slice {\em regular} function,
    if there exists such a corresponding stem function $F$ which is
    {\em holomorphic}.
  \end{enumerate}
\end{definition}

\subsubsection{$*$-product}
The space of slice regular functions on a axially
symmetric domain in $\Oct$
forms an alternative $\R$-algebra with the
$*$-product as multiplication.

This $*$-product may be defined by the correspondence with stem functions:

Given slice regular functions $f,g$ with stem functions $F$ resp.~$G$,
their ``star product'' is defined as the slice regular function
whose stem function is $F\cdot G$ (with $(F\cdot G)(z)=F(z)G(z)$).

If slice regular functions $f,g$ are described by convergent
power series
\[
f(q)=\sum_{k=0}^{+\infty} q^ka_k,\quad
g(q)=\sum_{k=0}^{+\infty} q^kb_k
\]
then 
\[
(f*g)=\sum_{k=0}^{+\infty} q^kc_k,\quad
c_k=\sum_{j=0}^ka_jb_{k-j}\quad
\text{(Cauchy product)}
\]
 {\em Warning:}
  In general $(f*g)(q)\ne f(q)g(q)$.
 
\subsection{Conjugation, norm and trace: Function algebras}\label{ss-3.4}

Given the notion of conjugation on an algebra we want to define
conjugation also on associated function algebras.
This is often intricate, since conjugation must be defined such
that the conjugate of a function is still a member of the function
algebra at hand.

For a {\em stem function} $F$ we define $(F^c)(z)=\left(F(z)\right)^c$,
i.e., we apply (octonion) conjugation pointwise.
Thus we obtain a {\em conjugation} on the algebra of stem functions
defined on a symmetric domain
$D\subset \C$.

For a {\em slice regular function} $f$ we may define its conjugate
$f^c$ using the
correspondence between slice functions and stem functions, i.e.,
given a slice regular function $f:\Omega_D\to\Oct$ with
stem function $F:D\to\Octc$ we define its conjugate $f^c$
as the slice regular function which has $F^c$ as stem function.

\begin{warning}
Given a slice regular function $f:q\mapsto f(q)$
in general neither $q\mapsto\overline{f(q)}$ nor
$q\mapsto \overline{ f(\bar q)}$ is regular.

In general we
have $f^c(q)\ne\overline{f(q)}$.
\end{warning}

From these definitions we easily deduce

\begin{lemma}
  \begin{enumerate}
  \item
    
    The map
    $f\mapsto f^c$ is an antiinvolution
    on the $\R$-algebra of slice regular functions on $\Omega_D$
    for every axially symmetric domain $\Omega_D$.
\item
  If $F=F_1+\iota F_2$ is the stem function for a slice regular function
  $f$, then
\begin{align*}\label{stem-conj}
  f(x+Iy)&=F_1(x+iy)+IF_2(x+iy)\\
  f^c(x+Iy)&=F_1^c(x+iy)+IF_2^c(x+iy)
\end{align*}
  \end{enumerate}
\end{lemma}

Once conjugation $f\mapsto f^c$ is defined, 
we define norm and trace in the usual way
(as in Definition~\ref{def-n-tr}).

For a stem function $F$ we obtain:
\[
\Nm(F)(z)=(FF^c)(z)=(F(z))(F^c(z))=(F(z))(F(z))^c=\Nm(F(z))
\]
and
\[
(\Tr  F)(z)=(F+F^c)(z)=F(z)+(F(z))^c=\Tr (F(z))
\]

\subsection{Power series}

Let $\Omega=B_r=\{q\in\Oct:||q||<r\}$ with $0<r\le +\infty$.
Then every slice regular function $f$ on $\Omega$ is given by
a power series $f(q)=\sum_{k=0}^{+\infty} q^ka_k$ ($a_k\in\Oct$)
which converges on all of $\Omega$.

In this case we have:
\[
f^c(q)=\sum_{k=0}^{+\infty} q^k\bar a_k,\quad\quad
(\Tr f)(q)=\sum_{k=0}^{+\infty} q^k\left(\Tr  a_k\right)
\]

\subsection{Conjugation, norm and trace: summary}

Immediate from the construction we obtain:

\begin{proposition}\label{n-tr-stem}
  Let $f$ be a slice function and $F$ its associated stem function.

  Then 
  $\Nm(F)$, $\Tr(F)$ and $F^ c$ are the stem functions associated to
  $\Nm(f)$, $\Tr(f)$ resp.~$f^c$.
\end{proposition}

We will apply the notions of conjugation, norm and trace
not only to $\Oct$, but 
to all of the following $\R$-algebras:

\begin{enumerate}
\item
  The algebra $\Oct$ of octonions with the usual multiplication
  and conjugation with $C=\R$.
\item
  The complexified octonions $\Octc$  with $C=\C$  embedded into
  $\Octc=\Oct\tensor_{\R}\C$ as $\R\tensor_{\R}\C$.
  Conjugation on $\Oct$ is a $\R$-linear self map of $\Oct$ which naturally
  induces a $\C$-linear self map on the tensor product
  $\Octc=\Oct\tensor_\R\C$. We take this 
  octonionic conjugation as the antiinvolution.
  This octonionic conjugation is not to be confused with the complex
  conjugation of the complex vector space   $\Octc=\Oct\tensor_\R\C$. 
\item
  The algebra of slice regular functions on an axially symmetric
  domain $\Omega_D$ 
  with the star product as multiplication and
  $f \mapsto f^ c$ (see discussion in \paragraph\ref{ss-3.4})
  as involution  and the subalgebra of slice preserving
    functions (see \paragraph\ref{subsect-sp}) as $C$.
\item
  The algebra of ``stem functions'' $F:D\to\Octc$
  on a symmetric domain $D$ with pointwise
  multiplication as product and pointwise octonionic
  conjugation as conjugation.
  Here $C$ denotes the subalgebra of those functions whose values
    are contained in the center of $\Octc$.
\end{enumerate}

\subsection{Slice preserving functions}\label{subsect-sp}

There is a special class of slice (regular) functions which
is called ``slice preserving''.

\begin{proposition}\label{eq-s-p}
  Let $D\subset \C$ be a symmetric domain with associated
  axially symmetric domain $\Omega_D$.
  
  Let $f:\Omega_D\to \Oct$ be a slice regular function with
  stem function $F:D\to\Octc$.

  Then the following are equivalent:
  \begin{enumerate}
  \item
    $f=f^c$,
  \item
    $F=F^c$,
  \item
    $F(D)\subset  \R\tensor_\R\C\subset \Oct\tensor_\R\C=\Octc$.
  \item
    $f(D\cap \C_I)\subset\C_I$ for all $I\in\S_{\Oct}=\{q\in \Oct:q^2=-1\}$
    (with $\C_I=\R+I\R$).
  \end{enumerate}
\end{proposition}

\begin{definition}\label{slice-preserv}
If one (hence all) of these properties are fulfilled,
$f$ is called ``{\em slice preserving}''.
\end{definition}

\begin{proof}
  These equivalences are well known. $(iv)\iff(iii)$ follows from
  representation formula.
  $(i)\iff(ii)\iff(iii)$ by construction  of $(\ )^c$.
\end{proof}

A slice regular function $f$ which is given by a convergent
power series $f(q)=\sum_{k=0}^{+\infty}q^ka_k$
is slice preserving
if and only if all the coefficients $a_k$ are real numbers.

\subsection{Compatibility}

Recall that conjugation for slice regular functions is {\em not}
just pointwise conjugation of the function values.

Therefore in general
\[
(\Tr f)(q)\ne \Tr(f(q)), (\Nm f)(q)\ne \Nm(f(q)).
\]

Let $B$ be a $\R$-subalgebra of an $\R$ algebra $A$.
Let $A$ be equipped with an antiinvolution which stabilizes
$B$.

Then for $x\in B$ the notions $\Tr(x)$ and $\Nm(x)$ are the same
regardless whether we regard $x$ as an element of $B$ or as an
element of $A$.

As a consequence
\begin{itemize}
\item
  For $x\in\Oct$ the notions $\Nm(x)$, $\Tr(x)$ agree
  independent of whether
  we consider $x$ in $\Oct$ or in $\Octc$.
\item
  For an element  $x\in\Oct$ the notions $\Nm(x)$, $\Tr(x)$ agree
  whether we regard $x$ in $\Oct$ or as a constant slice regular function
  with value $x$.
\end{itemize}

\subsection{Decomposing $\Oct$}
Frequently we will use the vector space decomposition $\Oct=C\oplus W$
where $C=\R$ is the center of $\Oct$ (and also
$C=\{x\in\Oct:x=\bar x\}$), and $W$ denotes the imaginary subspace,
i.e.,
\[
W=\ker\Tr=\{x\in\Oct:x=-\bar x\}
\]

This
decomposition $\Oct=C\oplus W$
is a vector space decomposition (in fact the eigenspace decomposition
for the conjugation map), but not an algebra decomposition.

It induces a similar decomposition $\Octc=C_{\C}\oplus W_{\C}$ of the
complexification.

\section{Strategy}
Here we want to present a rough ``road map'' for the proof of our main theorem
(Theorem~\ref{mainth}).

As always in this paper,
$\Oct$ denotes the algebra of octonions.

Let $\Omega_D$ be an axially symmetric domain in $\Oct$ associated to
a symmetric domain $D\subset\C$ (as in \paragraph\ref{sect-dom}).

The most difficult part of our main theorem (Theorem~\ref{mainth})
is the implication $(iii)\implies(iv)$.
In order to prove this
we have to show the following statement for
certain%
\footnote{namely, ``Stem functions'' as discussed in the preceding
section}
 holomorphic maps $F,H:\dom\to\Octc$
  
 \begin{center}\fbox{\begin{minipage}{9truecm}{\em If
       for every $z\in\dom$ there exists an element
       $g\in\GC=\Aut(\Octc)$ such that
       $F(z)=g\left(H(z)\right)$, then there exists a holomorphic map
       $\phi:\dom\to\GC$ with $F(z)=\phi(z) \left(H(z)\right)\ \forall z\in\dom$.
}
   \end{minipage}}
 \end{center}

This amounts to finding a section for a certain projection map, namely
$\pi:V\to\dom$
with
\[
V=\{(z,g)\in\dom\times\GC: F(z)=g\left(H(z)\right) \}
\]
and $\pi(z,g)=z$.

It is easy to reduce to the case where $\Tr(F)=\Tr(H)=0$
(Lemma~\ref{traceless}).

We discuss the locus where the imaginary parts of
$F$ and $H$ assume zero as value.
For this purpose we need the notion of a ``central divisor''
(\paragraph~\ref{ss-cdiv}).
Using this notion we arrive at the case where
the imaginary parts of $F$ and $H$ have no
zeroes. Then all the fibers of $\pi:V\to\dom$ have the same dimension.

Based on the construction of $V$ and an analysis of the automorphism
group $\Aut(\Octc)$ of the algebra of
complexified octonions $\Octc$,
we deduce that there is a discrete subset $L\subset\dom$ such that
$\pi$ restricts to a holomorphically locally trivial fiber bundle
over $\dom_0=\dom\setminus L$ (Proposition~\ref{loc-triv}).

For a  point $p$ in the discrete set $L$ the map $\pi$ is not
a locally trivial bundle on any neighborhood of $p$ in $D$;
in fact here the $\pi$-fiber is not
isomorphic to the generic $\pi$-fiber.

We continue as follows: First we construct a topological section
over $\dom_0$ (Proposition~\ref{top-section}),
then we show that we can extend this section to
a section defined on all of $D$ by suitable modifications near
the special fibers (Proposition~\ref{global-con}).
Here it is important that the generic fibers
are simply-connected, and that by a result 
in an earlier paper (see Proposition~\ref{loc-eq})
we know that $\pi$ admits everywhere local sections
(even holomorphic ones).

Once we have obtained a continuous section for $\pi:V\to\dom$, we may
deduce the existence of a {\em holomorphic} section $\sigma:D\to V$
using
Oka theory (Proposition~\ref{oka}).
For this purpose we verify that $\pi:V\to\dom$
is an {\em elliptic map} in the sense of Oka theory.

We would like to emphasize that
  for the case $A=\Oct$ we need new methods and technologies which
  differ from those we used in \cite{BWH} for $A=\H$.
  In particular, for $A=\H$ the fibers of the above mentioned
  map $\pi:V\to \C$ are one-dimensional, and in \cite{BWH} we used
  special properties only true in low dimensions. Here, for $A=\Oct$
  we need more general theory, in particular Oka theory.

On the other hand, the arguments we used for the octonion case
need that the generic fibers of the above mentioned map $\pi$
are simply-connected,
while in the quaternionic case they are isomorphic to $\C^*$. Thus the proof
for the quaternionic case is {\em not} a corollary to our result
for the octonion case.

\section{Local equivalence}

\begin{proposition}\label{loc-eq}
Let $G$ be a connected complex Lie group acting holomorphically on
  a complex manifold $X$ such that all the orbits have the same
  dimension $d$.

  Let
  \[
  F,H:\Delta=\{z\in\C:|z|<1\}
  \to X
  \]
  be holomorphic maps such that for every
  $z\in\Delta$ there exists an element $g\in G$ (depending on $z$,
  not necessarily unique)
  with $F(z)=g\cdot H(z)$.

Then 
 there exists $0<r<1$ and a holomorphic map
  \[
  \phi:
  \Delta_r=\{z:|z|<r\}\to G
  \]
  such that
  \[
  F(z)=\phi(z)\left( H(z)\right)\ \forall z\in\Delta_r.
  \]
\end{proposition}

\begin{proof}
  See \cite{BWH}, Proposition~9.4.
\end{proof}

\section{Topological preparations}    

\subsection{Existence of sections}

\begin{proposition}\label{top-section}
  Let $\pi:E\to B$ be a locally trivial topological
  fiber bundle and assume that
  the fiber $F$ is pathwise connected and that $B$ is homotopic
  to a (real) one-dimensional $CW$ complex.

  Then there exists a continuous section $\sigma:B\to E$.
\end{proposition}

\begin{proof}
    First we claim: {\em
  There is no loss in generality in assuming that $B$ is itself
  a one-dimensional $CW$-complex (instead of merely being homotopic to one.)}

  Indeed, let $W$ be a real one-dimensional $CW$-complex homotopic to $B$.
  This means that there are continuous maps $f:W\to B$ and $g:B\to W$
  such that $g\circ f$ and $f\circ g$ are homotopic to
  $id_W$ resp.~$id_B$.
  Assume that the pull back bundle $f^*E\to W$ has a section.
  This yields an induced section in $g^*(f^*E))\to B$
  (\cite{Steenrod}, Lemma~2.11).
  But the bundle $g^*(f^*E))$
  is isomorphic as a bundle to $E$, because
  $g\circ f\sim id_B$ (\cite{Steenrod},
  Theorem~11.5.)

  Thus from now one we may and do assume that $B$ itself is a real
  one-dimensional $CW$-complex.
  
  We remove points $p_i$ ($i\in I$)
  in the one-dimensional cells of $B$ such that
  the complement $M=B\setminus\{p_i:i\in I\}$
  has only contractible connected components.
  For this, we have to remove points in one-dimensional cells which are
  part of closed loops. Evidently this may be done by removing at
  most one point in each one-dimensional cell.
  For this reason the (possibly countably infinite)
  family of points $(p_i)_i$ forms a discrete subset of $B$.
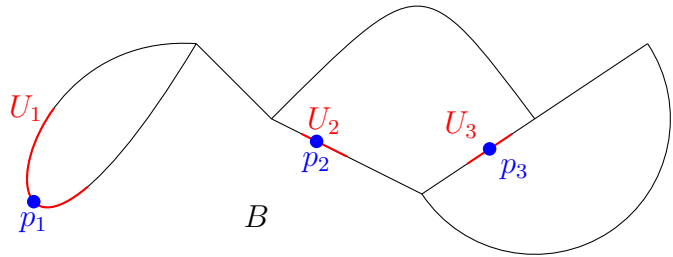
  \begin{center}
    \begin{tikzpicture}
      \draw [thin] (0,2) .. controls (-3,2.2) and (-3,-3) .. (0,2);
      \draw [thin] (0,2) -- (1,1) -- (3,0) -- (4.5,1) -- (6,2);
      \draw [thick,red] (1.4,0.8) -- node [above] {$U_2$} (2,0.5);
      \fill [blue] (1.6,0.7) node [below] {$p_2$} circle [radius=0.09];
      \draw [thick, red] (3.6,0.4) -- node [above left] {$U_3$} (4.2,0.8);
      \fill [blue] (3.9,0.6) node [below right] {$p_3$} circle [radius=0.09];
      \draw [thick, red] (-1.89,1.1432) node [left] {$U_1$} .. controls (-2.49,0.3252) and (-2.34,-0.6728) .. (-1.44,0.0992); %
      \fill [blue] (-2.16,-0.102) node [below] {$p_1$} circle [radius=0.09];
      \draw [thin] (3,0) arc [start angle=-146.3, end angle=33.7, radius=1.803];  
      \draw [thin] (1,1) .. controls (3,3) .. (4.5,1);
      \node at (0.8,-0.3) {$B$};
    \end{tikzpicture}
  \end{center}
  The restriction of the bundle to $M$ is trivial, because $M$
  is contractible. Hence there
  is a section $s_0:M\to E$.
  Each of the chosen points $p_i$ admits an open neighborhood
  $U_i$ with an homeomorphism  $\phi_i:U_i\,\simeq\,]-\!1,+1[$
      with $\phi_i(p_i)=0$.
      Fix $q_i^-=\phi_i^{-1}(-0.5)$ and
      $q_i^+=\phi_i^{-1}(0.5)$.
      
      The bundle admits a trivialization on $U_i$
      (because $U_i$ is contractible). Hence
      $\exists\alpha_i:\pi^{-1}(U_i)
      \stackrel{\sim}{\longrightarrow} U_i\times F$.
      Now we choose a path
      \[
      \gamma_i:[-0.5,+0.5]\to F
      \]
      with
      $ \alpha_i(s_0(q_i^-))=(q_i^-,\gamma_i(-0.5))$ and
      $ \alpha_i(s_0(q_i^+))=(q_i^+,\gamma_i(0.5))$.

      We obtain a section on $U_i$ as
      \[
      \sigma:x\mapsto
      \begin{cases} s_0(x) & \text{ if $x\not\in\phi_i^{-1}([-0.5,0.5])$}\\
        \alpha_i^{-1}(x,\gamma_i(t)) & \text{
          if $\phi_i(x)=t$ with $t\in[-0.5,0.5]$}\\
      \end{cases}
\]
      By performing this procedure around each $p_i$ and keeping $s_0$
      outside the union of all $U_i$ we obtain a global section
      $\sigma:B\to E$.
\end{proof}

\subsection{Homotopy equivalence of sections}
\begin{proposition}\label{app}
  Let $\pi:E\to B$ be a locally trivial
  topological
  fiber bundle with $B=S^1$
  where the fiber $F$ is connected and simply-connected.

  Then any two continuous sections are homotopic to each other.
\end{proposition}

\begin{proof}
  Let $\sigma_0,\sigma_1:B\to E $ be two continuous sections.

  Recall that  $S^1\simeq\R/\Z$.
  In particular, $S^1$ may be obtained from
  $I=[0,1]$ by identifying $0$ with $1$.
  Let $\rho:I\to \left(I/\!\sim\right) \simeq S^1=B$ be the corresponding quotient map.

  Let 
  \[
  \rho^*E=\tilde E\stackrel{\tilde\pi}\longrightarrow I
  \]
  be the pull-back of the bundle $\pi:E\to B$ under the map
  $\rho:I\to B$.

  Since $I$ is contractible, this pull-back bundle
  may be trivialized as described in the commutative
  diagram below

  \begin{center}
    \begin{tikzpicture}
      \matrix (m)
              [
                matrix of math nodes,
                row sep    = 3em,
                column sep = 4em
              ]
              {
               E && \tilde E && I\times F\\
               &\curvearrowright & &\curvearrowright &\\
               S^1\simeq B=I/\!\sim & & & I &\\
              };
              \path
              (m-1-3) edge [->] node [below] {$\phi=(\tilde\pi,\psi)$}
                node [above] {$\sim$} (m-1-5)
              (m-1-3) edge [->] node [below left] {$\tilde \pi$} (m-3-4)
              (m-1-5) edge [->] node [below right] {$pr_1$} (m-3-4)
              (m-1-1) edge [->] node [left] {$\pi$} (m-3-1)
              (m-1-3) edge [->] node [above] {$\tilde\rho$} (m-1-1)
              (m-3-4) edge [-> ] node [below] {$\rho$} (m-3-1);
;
    \end{tikzpicture}
  \end{center}
  Here $\tilde\rho:\rho^*E=\tilde E\to E$ and
  $\tilde \pi:\tilde E\to I$ are the
  natural maps which make the diagram commute and $\psi$ denotes the
  second component of the trivializing map $\phi:\tilde E\to I\times F$.

  Since $B$ is obtained from $I=[0,1]$ by identifying $0$ with $1$,
  we may obtain $E$ from $\tilde E\simeq I\times F$  by identifying
  $(0,p)$ with $(1,\alpha(p))$ for $p\in F$ where $\alpha$ denotes
  an homeomorphism of $F$. Then $\tilde\rho:\tilde E\to E$ is the
  natural quotient map.
  
  Let $\tilde \sigma_i:I\to\tilde E$ denote the pull-backs of the
  sections $\sigma_i$. Note that
  $\psi\left(\tilde\sigma_i(1)\right)=
  \alpha\left(\psi\left( \tilde\sigma_i(0)\right) \right)$
  for $i\in\{0,1\}$.

  Because $F$ is pathwise connected, there is a continuous map
  $\gamma:[0,1]\to F$ with
  \begin{align*}
    &\gamma(0) = \psi(\tilde\sigma_0(0)) \\
    \text{and }\quad&\gamma(1) = \psi(\tilde\sigma_1(0)). \\
  \end{align*}

  Let $Q=[0,1]\times[0,1]$.

    We define a map $K:\partial Q\to F$ as
    indicated in the diagram below.
    \begin{center}
    \begin{tikzpicture}[scale=2]
      \draw (0,0) node [below left] {$\scriptstyle (0,0)$} -- node [above] {$\gamma$}
      (1,0) node [below right] {$\scriptstyle (1,0)$} -- node [right] {$\psi\circ\tilde\sigma_1$}
      (1,1) node [above right] {$\scriptstyle (1,1)$} -- node [above] {$\alpha\circ\gamma$}
      (0,1) node [above left] {$\scriptstyle (0,1)$} -- node [left] {$\psi\circ\tilde\sigma_0$} (0,0);
    \end{tikzpicture}
  \end{center}
    
  Precisely, we define (for all $s,t\in[0,1]$):
    \begin{align*}
    K(0,t)&=\psi(\tilde\sigma_0(t)),\\
    K(1,t)&=\psi(\tilde\sigma_1(t)),\\
    K(s,0)&=\gamma(s)\\
    K(s,1)&=\alpha\left(\gamma(s)\right)\\
  \end{align*}

  It is easily checked that $K:\partial Q\to F$
  is well-defined and continuous.

  Since $F$ is simply-connected, we may extend $K$ to
  a continuous map $H_0:Q\to F$.

  Then there is a homotopy
  $H:[0,1]\times B\to E$
  from $\sigma_0$ to $\sigma_1$ via
  \[
  H(s,\rho(t))=\tilde\rho\left( \phi^{-1} \left( t, H_0(s,t)
  \right)\right)
  \]

    First, let us check that $H$ is well-defined. Since $\rho(0)=\rho(1)$,
  we need
  \[
  \forall s\in[0,1]:
  \tilde\rho( \phi^{-1} ( 0, \underbrace{H_0(s,0)}_{=\gamma(s)}
  ))
  =
  \tilde\rho( \phi^{-1} ( 1, \underbrace{H_0(s,1)}_{=\alpha(\gamma(s))}
  ))
  \]
  which is true because the definition of $\tilde\rho$ and $\alpha$
  implies
  \[
  \tilde\rho \left(  \phi^{-1} (0,x)\right)
  =
  \tilde\rho \left(  \phi^{-1} (1,\alpha(x))\right)\quad\forall x\in[0,1]
  \]

  Second, we verify that $H$ is indeed a homotopy from $\sigma_0$
  to $\sigma_1$:
    \[
  H(0,\rho(t))=\tilde\rho\left( \phi^{-1} \right. ( t,
  \underbrace{H_0(0,t)}_{=K(0,t)}
  ) \left. \hbox to 0pt{\phantom{$\phi^{-1}$}}\right)
  =\tilde\rho(
  \underbrace{\phi^{-1} \left( t,\psi(\tilde\sigma_0(t))
  \right)}_{=\tilde\sigma_0(t)}
  )=\sigma_0(t)
  \]
  and
  \[
  H(1,\rho(t))=\tilde\rho\left( \phi^{-1}\right. ( t,
  \underbrace{H_0(1,t)}_{=K(1,t)}
  ) \left. \hbox to 0pt{\phantom{$\phi^{-1}$}}\right)
  =\tilde\rho(
  \underbrace{\phi^{-1} \left( t,\psi(\tilde\sigma_1(t))
  \right)}_{=\tilde\sigma_1(t)}
  )=\sigma_1(t).
  \]
  
\end{proof}

\subsection{Existence of global sections}

In Proposition~\ref{top-section}
  we proved the existence of global continuous sections for certain
  {\em locally trivial} fiber bundles.
  For our purposes everywhere locally trivial bundles are a too restricted
  class of surjections, we need the existence of global continuous sections
  under weaker conditions. Hence we deduce the proposition below.

\begin{proposition}\label{global-con}
  Let  $E$ connected real manifold,
  $X$ non-compact Riemann surface, $L\subset X$ a discrete subset
  and $\pi:E\to X$ be a
  surjective $C^1$-map such that
  \begin{enumerate}
  \item
    There are local continuous sections everywhere,
    i.e., for every $x\in X$ there
    is an open neighborhood $U$ and a continuous map $\sigma:U\to E$
    with $\pi\circ\sigma=id_U$.
   \item
     The restriction of $\pi$ to $X_0=X\setminus L$ is
     a locally trivial fiber bundle with a connected and
     simply-connected fiber $F$.
  \end{enumerate}

  Then there exists a global continuous section $\sigma:X\to E$.
\end{proposition}

\begin{proof}
  
  $F$ is a connected manifold and therefore pathwise connected.
  $X$ is a non-compact Riemann surface and therefore homotopic to
  a one-dimensional $CW$-complex
  (see e.g.~\cite{Hamm83}).
  Hence Proposition~\ref{top-section} implies the existence
  of a continuous section  $\sigma_0$ on $X_0=X\setminus L$.
  
  For every $p\in L$ we choose an open neighborhood $U_p$ of $p$ in $X$
  such that
  \begin{enumerate}[label=(\arabic*)]
  \item
    all the $U_p$ are disjoint.
  \item
    There is a biholomorphic map
    $\zeta_p:U_p\to\Delta_3=\{z\in\C:|z|<3\}$ with
    $\zeta_p(p)=0$.
  \item
    There is a continuous section $s_p$ of $\pi$ on $U_p$, i.e.,
    a continuous map $s_p:U_p\to E$ with $\pi\circ s_p=id_{U_p}$.
  \end{enumerate}

  We fix $p\in D$, such a  map $\zeta_p$ and such a section $s_p$.

  By assumption $(ii)$, $\pi:E\to X$ restricts to a locally trivial
  bundle
  \[
  E\supset\pi^ {-1}\left(U_p\setminus\{p\}\right)
  \stackrel{\pi}\longrightarrow U_p\setminus p.
  \]

  Now for $r\in[1,2]$ we have maps $\xi_r:S^1\to U_p$
  defined as
  \[
  \xi_r(x)=\zeta_p^ {-1}(rx)\quad (x\in S^ 1=\{z\in\C:|z|=1\}).
  \]

  Since $\xi_r$ ($r\in[1,2]$) are all homotopic, the pull-back bundles
  $\xi_r^ *E$ are isomorphic.
  Fixing such an isomorphism, we may regard
  $s_p\circ\xi_1$ and $\sigma_0\circ\xi_2$
  as sections in the same fixed $F$-fiber bundle over $S^ 1$. Due to
  Proposition~\ref{app}
  it follows that there is a homotopy between
  $s_p\circ\xi_1$ and $\sigma_0\circ\xi_2$.
  
  Using the aforementioned isomorphism
  this homotopy yields a
  continuous map
  $H:[1,2]\times S^ 1\to E$ such that
  
  \begin{itemize}
  \item
    $\zeta_p(\pi \left(H(r,t)\right))=rt$.
  \item
    $H(1,t)=s_p\left( \zeta_p^{-1}(t)\right)$.
  \item
    $H(2,t)=\sigma_0\left( \zeta_p^{-1}(2t)\right)$.
  \end{itemize}

  Now we may define $\sigma$ on $U_p$ as
  \[
  \sigma(x)=\begin{cases}
  s_p(x) & \text{ if $|\zeta_p(x)|\le 1$}\\
  H(r,t) & \text{ if $r=|\zeta_p(x)|\in[1,2]$ with
    $\zeta_p(x)=rt$, $t\in S^1$}\\
    \sigma_0(x) & \text{ if $|\zeta_p(x)|\ge 2$}\\
  \end{cases}
  \]

  Since we may do this at every point $p\in D$ independently, we
  obtain a globally defined continuous section.
\end{proof}

\section{Oka Theory}

In complex analytic geometry there is the notion of an {\em Oka manifold}.
If a complex manifold $X$ is a Oka manifold, then for every Stein
manifold $Z$ and every continuous map $f_0:Z\to X$ there
exists a holomorphic map $f:Z\to X$ which is homotopic to $f_0$.

See \cite{F2} for more information about Oka manifolds.

\subsection{Elliptic Maps}

In Oka theory, there is the notion of {\em `` elliptic''} maps
(\cite{F2}, Definition~6.1.2) which we will use.

\begin{definition}\label{def-elliptic}%
\footnote{The definition as stated here is more restrictive
    than the original one in \cite{F2}, Definition~6.1.2. We do not
    need the most general form.}
  Let $f:X\to Y$ be a holomorphic map between complex
  manifolds.

  The map $f$ is called {``\em elliptic''} if there
  exists a 
    {``\em (dominating fiber) spray''}, i.e., if there exists
    
  a holomorphic vector bundle $\pi:E\to X$ and a
  holomorphic map $s:E\to X$ satisfying the following properties:
  
  \begin{center}
    \begin{tikzpicture}[scale=0.8]
      \draw [->] (0,1.7) -- node [left] {$s$} (0,0);
      \draw [->] (0.5,1.7) -- node [right] {$\pi$} (0.5,0);
      \node at (0.25,-0.5) {$X$};
      \node at (0.25,2.2) {$E$};
      \draw [->](0.7,-0.5) -- node [below] {$f$} (2.7,-0.5);
      \node at (3.2,-0.5) {$Y$};
    \end{tikzpicture}
  \end{center}
  
  \begin{itemize}
  \item
    For every $x\in X$ let $0_x$ denote the point in the zero-section of
    the vector bundle $E$ which is above $x$.
    Then $s(0_x)=x\ \forall x\in X$.
  \item
    $\forall p\in Y:\ s(E_p)\subset X_p$ for
    $X_p=f^ {-1}(p)$ and   $E_p=(f\circ\pi)^ {-1}(p)=\pi^ {-1}(X_p)$.
    In other words: $f\circ s=f\circ\pi$.
  \item
    For every $x\in X$,
    $V_x=\{v\in T_xX:(Df)_x(v)=0\}=T_x\left(X_{f(x)}\right)$,
    $W_x=\{w\in T_{0_x}E: (D\pi)_{0_x}(w)=0\}=T_{0_x}\left(\pi^ {-1}(x)\right)$
    the linear map
    $(Ds)_{0_x}:W_x\to V_x$
    is surjective.
  \end{itemize}
\end{definition}

\begin{remark}
  If $f$ is constant, the last condition is equivalent to $s$
  being submersive at every point of the zero-section.
\end{remark}

\begin{example}
  Let $f:X\to Y$ be an unramified covering. Then $f$ is an {\em elliptic
    map} (using as $E$ the trivial bundle with fiber $\{0\}$).
\end{example}

For us, the important fact on elliptic maps is the following:

\begin{theorem}\label{thm-elliptic}
  Let $f:X\to Y$ be an elliptic holomorphic map.
  Assume that $Y$ is a Stein manifold.

  Then every continuous section $\sigma$ (i.e.~every continuous map
  $\sigma:Y\to X$ with $f\circ\sigma=id_Y$)
  is homotopic to a holomorphic section.
\end{theorem}

\begin{proof}
See
\cite{F2}, Theorem~6.2.2..
\end{proof}

\begin{example}
Let $G$ be a complex Lie group, $p:P\to B$ a $G$ principal
bundle. Then the projection map $p:P\to B$ is an {\em elliptic} map:

We set $f=p$, $X=P$, $Y=B$, $E=P\times \Lie(G)$
(where $\Lie (G)$ denotes the Lie algebra of $G$);
$\pi$ denotes the projection from $E$ to the first factor. In this way
$E$ is a trivial vector bundle with fiber $\Lie(G)$ over $B$.
Let $\mu:P\times G\to P$ be the principal right action
of the principal bundle $p:P\to B$. Then we may choose
$s$ as
\[
s:P\times\Lie(G)\ni (x,v)\mapsto \mu(x,\exp(v))
\]
In this way, the above Theorem~\ref{thm-elliptic} generalizes
the classical Grauert Oka theorem (\cite{G}).
\end{example}

\begin{proposition}\label{loc-triv-ell}
  Let $G$ be a complex Lie group acting transitively on
  a connected complex manifold $F$.
  
  Let $\pi:H\to B$ be a holomorphic locally trivial fiber bundle
  with fiber $F$ and structure group $G$.

  Then $\pi$ is an elliptic map.
\end{proposition}

\begin{proof}
  Let $\cU=\left(U_i\right)_ {i\in I}$ be a trivializing open cover
  on $B$.
  We may identify $H$ with the quotient of
  \[
  \{ (x,i,p):i\in I, x\in U_i,p\in F\}
  \]
  with respect to the equivalence relation
  \[
  (x,j,p)\sim (x,i,\phi_{ij}(x)(p))
  \]
  for some {``transition functions''} $\phi_{ij}:U_i\cap U_j\to G$.

  Then we consider the quotient $E$ of
  \[
  \{(x,i,p,v):i\in I, x\in U_i, p\in F, v\in\Lie(G)\}
  \]
  with respect to
  \[
  \left((x,i,p,v)\sim (x,i,\phi_{ij}(x)(p),Ad\left(\phi_{ij}(x)\right)(v)\right)
  \]

  Since the adjoint action of $G$ on its Lie algebra $\Lie(G)$ is linear,
  the natural projection $E\to H$ is a vector bundle.

  Finally we define a {\em spray}
  $s:E\to H$ by a representative:
  \[
  s:[(x,i,p,v)]\mapsto [(x,i,\exp(v)(p))]
  \]
  where $[(x,i,p,v)]$ resp.~$[(x,i,\exp(v)(p))]$ denotes the point in $E$
  resp.~$H$ represented by $(x,i,p,v)$ resp.~$(x,i,\exp(v)(p))$.
  
  It is easily verified that this is well-defined and indeed
  a dominating spray.
\end{proof}

  \begin{proposition}\label{oka}
    Let $\Gc$ be a complex Lie group acting holomorphically on a
    complex manifold $X$. Assume that all isotropy groups
    have the same dimension $k$. Let $Z$ be a Stein complex manifold.
    Let $C,F:Z\to X$ be holomorphic maps and let
    \begin{equation}\label{V-def}
    V=\{(g,t)\in \Gc\times Z: g \left( C(t)\right)=F(t)\}
    \end{equation}

    Let $\pi:V\to Z$ be the natural projection map:
    $\pi(g,t)=t$.
    Then $\pi$ admits a holomorphic section
    $\sigma:Z\to V$ if and only if it admits a
    continuous section.

        \begin{center}
    \begin{tikzcd}
  \Gc\times Z & V \arrow[l, phantom, sloped, "\supset"]
                  \arrow[d,"\pi"] & \\
                  & Z \arrow[r,bend left,"C"]\arrow[r,bend right,"F"']
                  \arrow [u,bend left,dashed,"\sigma"] & X \\
    \end{tikzcd}
    \end{center}
   \end{proposition}

  \begin{proof}
    First we observe that $\pi\circ\sigma=id_Z$ implies the surjectivity
    of $\pi$.
    Thus the statement is trivially true if $\pi$ is not surjective:
    Without surjectivity of $\pi$ neither continuous nor holomorphic
    sections may exist.
    
    From now on  we assume that $\pi:V\to Z$ is surjective.
    Then there exist local holomorphic sections due to
    Proposition~\ref{loc-eq}.
    Let $p\in Z$ and let $\sigma:W\to V$ be a local section
    in an open neighborhood $W$ of $p$ in $Z$.
    Then $\pi\circ\sigma=id_W$, implying
    $D(\pi\circ\sigma)_p=id$.
    It follows that $D\pi$ is surjective, i.e.,
    $\pi$ is submersive.
    
    Let $\Gamma=Gm_k(\Lie \Gc)$ be the Grassmann manifold
    para\-me\-tri\-zing
    $k$ dimensional vector subspaces of the Lie algebra of $\Gc$.

    Recall that all the isotropy groups have the same dimension $k$.
    Thus we have a map $\zeta$ from $Z$ to $\Gamma$ mapping
    a point $t\in Z$ to the
    point in the Grassmann manifold corresponding to the
    Lie algebra of the
    isotropy group at $F(t)$.
    
    We recall the notion of the {``\em tautological vector bundle''}
    $\rho:\Theta\to\Gamma$  which is defined as
    \[
    \Theta=\{ ([U],u)\in \Gamma\times \Lie \Gc: u\in U\},
    \quad \rho([U],u)=[U]
    \]
    
    Let $E=(\zeta\circ\pi)^*\Theta$ be the pull-back as a vector bundle, i.e.,
      
    \begin{align*}
      E&=\{(\vartheta,v)\in\Theta\times V:
      \rho(\vartheta)=\zeta(\pi(v))\}\\
      &
      \simeq
      \left\{([U],u;v)\in\Gamma\times\Lie\Gc\times V
      :\ u\in[U]=\zeta(\pi(v))   \right\}\\
      &
      =
      \left\{([U],u;v)\in\Gamma\times\Lie\Gc\times V
      :\ u\in[U],\ U=\Lie \left(\Gc\right)_{F(\pi(v))}
      \right\}
    \end{align*}
    The condition $U=\Lie \left(\Gc\right)_{F(\pi(v))}$ implies
    that $[U]$ is determined by $v$. Hence
    \[
    E\simeq  \left\{(u;v)\in\Lie\Gc\times V
      :\ u\in\Lie \left(\Gc\right)_{F(\pi(v))}
      \right\}
    \]

    We recall the definition of $V$ as a subset of $\Gc\times Z$
    as in \eqref{V-def}
     and observe that
      \begin{align*}
      &u\in\Lie \left(\Gc\right)_{F(\pi(v))}\\
        \iff &\exp(ru) \in \left(\Gc\right)_{F(\pi(v))}\ \forall r\in\C\\
        \iff & \exp(ru)(F(\pi(v)))=F(\pi(v))\ \forall r \in\C\\
      \end{align*}
      Therefore:
      \begin{equation}\label{E-tripel}
        \begin{minipage}{0.5\textwidth}
          \[
                    E\simeq \{ (u,g,t)\in \left(\Lie\Gc\right)\times \Gc\times Z
              : (g,t)\in V ,
                    \]\vskip -20pt
                    \[
                    \exp(ru)(F(t))=F(t)\ \forall
    r\in\C \}.
    \]
                \end{minipage}
    \end{equation}
    Let 
    $\tau:E\to V$ be the natural projection onto $V$.
    Now $E\to V$ is a vector bundle such that the fiber over a point
    $(g,t)\in V$ is naturally isomorphic to the Lie algebra of
    the isotropy group for the $\Gc$-action on $X$ at $F(t)$, i.e.,
    \begin{center}
    \begin{tikzpicture}
     \matrix (m)
    [
      matrix of math nodes,
      row sep    = 3em,
      column sep = 4em
    ]
    {
      E && \Theta \\
      V & Z & \Gamma \\
    };
    \path
    (m-1-1) edge [->] (m-1-3)
    edge [->] node [left] {$\tau$} (m-2-1)
    (m-1-3) edge [->] node [right] {$\rho$} (m-2-3)
    (m-2-1) edge [->] node [below] {$\pi$} (m-2-2)
    (m-2-2) edge [->] node [below] {$\zeta$} (m-2-3)
    ;
    \end{tikzpicture}
\end{center}
    
    We define a spray $s$ as follows:
    \[
    E\ni(u,g,t)\stackrel{s}{\mapsto} \left( \exp(u)\cdot g, t\right)\in V
    \]

    Let us verify that $s$ is in fact a dominating fiber spray:
    
    \begin{itemize}
    \item
      By definition of $E$ we have
      $(u,g,t)\in E \implies u\in\Lie\Gc\implies \exp(u)\in\Gc$
      and
      $\exp(ru)(F(t))=F(t)\ \forall r\in\C$.
      By definition of $V$ in \eqref{V-def}
      for $(g,t)\in V$ we obtain
      \[
      g(C(t))=F(t)
      \]
      With
      $\exp(ru)\left(F(t)\right)=F(t)\ \forall r$
      this implies
      \[
      F(t)=\exp(ru)\left(F(t)\right)=\exp(ru)\left(g(C(t))\right)
      =\left(\exp(ru)\cdot g\right)\left(C(t)\right)
      \]

    which (specializing to the case $r=1$)
    implies
    \[
    (\exp(u)\cdot g,t)\in V.
    \]
      Thus $s$ defined as above is indeed a map from $E$ to $V$.

    \item
      For $u=0$ we have $\exp(0)=e_{\Gc}$ and therefore
      $\exp(0)\cdot g=g$.

      Thus
      \[
      s(0,g,t)=\left((\exp(0))\cdot g,t\right)=(g,t)
      \]
    \item
      \[
      \pi(s(u,g,t))=\pi(\exp(u)\cdot g, t)=t=\pi(g,t)
      =\pi\left(\tau(u,g,t)\right)
      \]
      Therefore $\pi\circ s=\pi\circ\tau$.
      
    \item
            Fix $(g,t)\in V$.
      We consider the space of ``vertical vector fields''
      \[
      W=\left\{ w\in T_{(0,g,t)}E : (D\tau)(w)=0\right\}.
      \]
      We have to show that $Ds$ maps $W$ surjectively onto
      \[
      W'=\left\{ w\in T_{(g,t)}V: (D\pi)(w)=0\right\}
      \]
      Let $H=(\Gc)_{F(t)}$. The exponential map from the Lie algebra
      $\Lie H$ to the Lie group $H$ is a diffeomorphism near $0\in\Lie H$.
      
      Furthermore, standard Lie group theory tells us that
      for every action of a Lie group $H$ on a manifold the fundamental
      $H$-vector fields span the tangent space of each $H$-orbit
      everywhere.

            Observe that the fiber
      \[
      V_t=\pi^{-1}(t)=\{(g,t):g\in \Gc: g \left( C(t)\right)=F(t)\}
      \]
      admits a natural transitive action of $H=(\Gc)_{F(t)}$
      given by
      \[
      H\ni h:(g,t)\mapsto (hg,t)
      \]
      The $H$-fundamental vector fields for this action
      span the tangent space of $V_t$ everywhere.
      Since $s(u,g,t)=\left(\exp(u)\cdot g,t\right)$,
      it follows that
      $Ds$ maps $W$ surjectively onto $T_{(g,t)}V_t=W'$.
          \end{itemize}

    Thus $\pi:V\to Z$ is elliptic in the sense of \cite{F2},
    Definition 6.1.2.

    Now \cite{F2}, Theorem~6.2.2 implies that there is a
    weak homotopy equivalence between the space of continuous sections
    and the space of holomorphic sections.
    This implies the assertion.
  \end{proof}

\section{Existence of holomorphic sections}

\begin{proposition}\label{section-2-con}
  Let $X$ be a non-compact Riemann surface.
  Let $S$ be a
  (not necessarily connected) complex Lie group and let
  $F$ be a connected complex manifold on which $S$ acts
  transitively and let
  $\pi:E\to X$ be a locally trivial holomorphic fiber bundle
  with fiber $F$ and structure group $S$.

  Then $\pi$ admits a global holomorphic section.
\end{proposition}

\begin{proof}
  First we recall that a non-compact Riemann surface is homotopic
  to a real one dimensional $CW$-complex (\cite{Hamm83}).
  Thus the existence of  a continuous section follows from
  Proposition~\ref{top-section}.

  Due to Proposition~\ref{loc-triv-ell}
  $\pi:E\to X$ is an elliptic map.
  Furthermore $X$ is  Stein, because it is a non-compact Riemann surface.
  Hence Theorem~ \ref{thm-elliptic} implies the existence
  of a global holomorphic section.
\end{proof}

\section{Automorphisms of  Octonions}
\subsection{Automorphisms of $\Oct$}

\begin{proposition}\label{auto-ortho}
  Every ring automorphism of $\Oct$ is $\R$-linear, continuous,
  commutes with conjugation and preserves the scalar product given
  as
  \[
  \left<q,r\right>=\Re\left(q\bar r\right)
  \]
\end{proposition}

  \begin{proof}
See \cite{B}\cite{SV}.
  \end{proof}
  
However, unlike in the case of the quaternionic algebra
$\H$, not every
orientation preserving orthogonal linear map fixing $\R$
is an $\R$-algebra automorphism of $\Oct$
(see Example~\ref{orth-not-auto} below).

In the next subsection, we present
a precise description of the automorphism group of $\Oct$,
using the theory of {\em ``basic triples''}.

\subsection{Basic triples}

\begin{definition}\label{def-basic-triple}
  A {\em``basic triple''}
  is an ordered triple of elements
  $e_1,e_2,e_3\in\Oct$ 
  such that

  \begin{enumerate}
  \item
    $||e_k||=1\ \forall k$.
  \item
    every $e_k$ is purely imaginary.
  \item
    $e_1$ and $e_2$ are orthogonal.
  \item
    $e_3$ is orthogonal to $e_1$, $e_2$ and $e_1e_2$.
  \end{enumerate}
\end{definition}

  \begin{theorem}\label{simple-triple}
    The automorphism group of $\Oct$ acts
    {\em simply transitively}\footnote{i.e., for every $x,y$ there is a
      {\em unique} element $g$ mapping $x$ to $y$.}
    on the set of basic triples.
  \end{theorem}

  \begin{proof}
See \cite{B}\cite{SV}.
  \end{proof}

\begin{corollary}\label{oct-iso}
  Let $q\in \Oct\setminus\R$. Then the isotropy group
  \[
  I=\{\phi\in\Aut(\Oct):\phi(q)=q\}
  \]
  is isomorphic to $SU_3(\C)$.
\end{corollary}

\begin{proof}
  Let $e_1$ be a purely imaginary element of $\Oct$
  with $||e_1||=1$
  such that $q=r+te_1$ for some $r,t\in\R$, $t\ne 0$.

  Then the isotropy at $q$ equals the isotropy at $e_1$.
  We choose $e_2,e_3$ such that $(e_1,e_2,e_3)$ is a basic triple.

  Given $e_1$, we have to choose $e_2$ in a $5$-dimensional sphere and
  then $e_3$ in a three-dimensional sphere.
  Since $\Aut(\Oct)$ acts simply transitively on the set
  of basic triples, $I$ can be identified with the set
  of basic triples with fixed $e_1$. It follows that $\dim_\R(I)=8$.

  Note that $e_1\cdot e_1=-1$. Because $\Oct$ is an alternative algebra,
  this implies
  \[
  e_1\cdot(e_1\cdot x)=(e_1\cdot e_1)\cdot x=-x\ \forall x\in \Oct.
  \]

  Let $P$ denote the orthogonal complement of $\left<1,e_1\right>$ in $\Oct$.
  Then $J:x\mapsto e_1\cdot x$ defines a complex structure on $P$.
  Note that every $\phi\in I$ acts
  trivially on $\left<1,e_1\right>$ and therefore
  stabilizes $P$.
  
  Observe that
  \[
  \phi(e_1\cdot x)=\phi(e_1)\cdot\phi(x)=e_1\cdot\phi(x)
  \ \forall x\in P,\phi\in I
  \]
  i.e., $\phi$ commutes with left multiplication by $e_1$.
  This means that $\phi$ is a {\em unitary} transformation with
  respect to the complex structure on $P$ defined by
  $J:x\mapsto e_1\cdot x$.

  Furthermore we note that $\det\phi=1$, because $\Aut(\Oct)$ is simple
  (Theorem~\ref{aut-g2}).
  Since $\phi$ acts trivially on $\left<1,e_1\right>$, it follows that
  $\det(\phi|_P)=1$.

  Therefore
  \[
  I\subset SU(P)\simeq SU_3(\C).
  \]
  For dimension reasons (both $I$ and $SU_3(\C)$ are real
  $8$-dimensional) we have equality.
\end{proof}
    
  \begin{corollary}\label{n-tr-auto}
    Let $p,q\in \Oct$.

    Then the following is equivalent;
    \begin{itemize}
    \item
      $N(p)=N(q)$ and $Tr(p)=Tr(q)$.
    \item
      There is an $\R$-algebra automorphism $\phi$ of $\Oct$
      such that
      $\phi(p)=q$.
    \end{itemize}
  \end{corollary}

      {\em Remark:} A similar statement is to be found
      in \cite{DentoniSce}, $L_4$, p.260. For the convenience
      of the reader we nevertheless provide a proof.
      
  \begin{proof}
    For every $q\in \Oct$ let $q_0$ denote the real
    part and
    let $q_v$ denote the imaginary (sometimes called: vectorial) part,
    i.e., $q=q_0+q_v$ with $q_0\in\R$ and $q_v=-\bar q_v$.
    Then $Tr(q)=2q_0$ and $N(q)=||q||^2=||q_0||^2+||q_v||^2$.

    This implies: If $Tr(p)=Tr(q)$ and $N(p)=N(q)$, then
    $p_0=q_0$ and $||p_v||=||q_v||$.

    If $||p_v||=||q_v||=0$, then $p=p_0=q_0=q$ and we may take the
    identity map as $\phi$.
    Thus we may assume that $||p_v||=||q_v||>0$.

    Define
    \[
    \tilde p_v=\frac{p_v}{||p_v||},\ \ \tilde q_v=\frac{q_v}{||q_v||}
    \]
    We may complete $\{\tilde p_v\}$ resp.~$\{\tilde q_v\}$
    to a basic triple of $\Oct$.
    Now Theorem~\ref{simple-triple} implies that there exists
    an automorphism $\phi$ of $\Oct$ with $\phi(\tilde p_v)=\tilde q_v$.
    Since $\phi$ is linear, and $||p_v||=||q_v||$, it follows
    that $\phi(p_v)=q_v$.
    Because $\phi$, like every algebra automorphism of $\Oct$, acts trivially
    on the center $\R$, we may conclude that $\phi(p)=\phi(q)$.

    For the converse, let $\phi\in\Aut(\Oct)$. Then $\phi$ acts linearly
    on $W$ and trivially on the center $\R$. As a consequence, $\phi$
    commutes with conjugation. Due to the definition of $N$ and $Tr$
    this implies that $Tr(\phi(x))=Tr(x)$ and $N(\phi(x))=N(x)$
    for all $x\in \Oct$.
  \end{proof}

  \begin{example}\label{orth-not-auto}
    Let $(e_1,e_2,e_3)$ be a basic triple for $\Oct$, and let
    $P$ be the orthogonal complement of
    $\left<1,e_1,e_2,e_3\right>$ in $\Oct$.
    Since $\dim_{\R}(P)=4$ and therefore $SO(P)\ne\{Id\}$, there is
    a non-trivial orientation preservation orthogonal transformation
    $\phi_0$ on $P$. It extends to a bilinear self-map $\phi\in SO(\Oct)$
    with $\phi|_{P}=\phi_0$, $\phi|_{\R}=id$ and $\phi(e_k)=e_k$ ($k=1,2,3$).
    By Theorem~\ref{simple-triple}, an $\R$-algebra automorphism of $\Oct$
    preserving $e_1$, $e_2$ and $e_3$ must be the identity map.
    Thus $\phi$ is an orientation preserving ortho\-gonal transformation
    on $\Oct$ which is not an $ \R$-algebra automorphism.
  \end{example}
  
  \begin{theorem}\label{aut-g2}
    The automorphism group of the $\R$-algebra $\Oct$ is a
    simply-connected compact simple real Lie group of type $G_2$.
  \end{theorem}

  \begin{proof}
    See \cite{B}.
  \end{proof}

      \begin{theorem}\label{autc-g2}
    The automorphism group of the $\C$-algebra $\Octc$ is a
    connected complex simple Lie group of type $G_2$.
    \end{theorem}

    \begin{proof}
      See  \cite{SV}, Theorem 2.3.5. %
    \end{proof}
      
    \begin{corollary}
      Let $G$ be the automorphism group of the real algebra $\Oct$
      and let $\Gc$ be the automorphism group of the
      $\C$-algebra $\Octc$ and consider the induced action of $G$
      on $\Octc$.

      Then $\Gc$ is the smallest complex Lie subgroup of
      $GL_\C(\Octc)$ containing $G$.
    \end{corollary}

    \begin{proof}
      Let $H$ be the smallest complex Lie subgroup
      containing $G$. Since $\Gc$ is a complex Lie group,
      $H\subset\Gc$. On the other hand, 
      $G$ is totally real and $\dim_\R(G)=14=\dimc(\Gc)$.
      Hence $\dim_\C(H)=14=\dim_\C(\Gc)$. In combination
      with the connectedness of $\Gc$ and $H\subset\Gc$
      this implies $H=\Gc$, i.e., $\Gc$ is the smallest complex
      Lie subgroup of $GL_\C(\Octc)$ containing $G$.
    \end{proof}

   \section{Orbits in the complexified algebra}

  The proposition below is principally applied to the situation,
  where $A=\Oct$ and $A=\R\oplus W$ as vector space,
  $W$ being the subspace of totally imaginary elements.

  In \cite{BWH} (Proposition~12.1)
  we deduced the following proposition:
  
  \begin{proposition}\label{VC-orbits}
  Let $W=\R^n$, and let $G$ be a connected real Lie group
  acting by orthogonal linear transformations on $W$ such that
  the unit sphere $S=\{v\in\R^n:||v||=1\}$ is a $G$-orbit.

  Let $W_{\C}=W\tensor_\R\C$. Let $B$ denote the $\C$-bilinear form
  on $W_{\C}$ extending the standard euclidean scalar product on $W=\R^n$.
  
  Let $\GC$ be the smallest complex Lie subgroup of $GL(W_{\C})$
  containing $G$. Then the $\GC$-orbits in $W_{\C}$ are the
  following:

  \begin{itemize}
    \item
    $H_\lambda=\{v\in W_{\C}:B(v,v)=\lambda\}$ for $\lambda\in\C^*$.
  \item
    $H_0=\{v\in W_{\C}:B(v,v)=0\}\setminus\{0\}$.
  \item
    $\{0\}$.
  \end{itemize}
\end{proposition}

  We also need some corollaries of this proposition,
    likewise proved in \cite{BWH}..

\begin{corollary}\label{eq-iso-dim} ($=$Corollary 12.2 of \cite{BWH}).
  Let $A$ be a finite-dimensional $\R$-algebra with $\R$ as center.
  Let $A=\R\oplus W$ as vector space and let $G$ be a real Lie group
  acting trivially on $\R$ and by orthogonal linear transformations on $W$.
  Assume that $G$ acts transitively on the unit sphere of $W$.

  Let $\GC$ be the smallest complex Lie subgroup of
  $GL(A\tensor_\R\C)$ containing $G$.

  Then all the $\GC$-orbits in
  $\left(W\tensor_\R\C\right)\setminus\{0\}$
  are complex hypersurfaces. In particular, they all
  have the same dimension.
\end{corollary}

\begin{corollary}\label{dim-c-isotropy} ($=$Corollary 12.3 of \cite{BWH}).
  Under the assumptions of Corollary~\ref{eq-iso-dim}
  for every point $p\in\left(W\tensor_\R\C\right)\setminus\{0\}$
  the isotropy group $I$ of the $\GC$-action at $p$ satisfies
  \[
  \dim_\C I+\dim_\R W-1=\dim_\C(\GC).
  \]
\end{corollary}

\begin{corollary}\label{omegaconj}  ($=$Corollary 12.4 of \cite{BWH}).
  Under the same assumptions, there is a Zariski open subset
  \[
  \Offen\subset A\tensor_\R\C\stackrel{\zeta}
  \sim\C\oplus \left(W\tensor_\R\C\right)
  \]
  defined as $\Offen=\{q:\zeta(q)=(x,v), B(v,v)\ne 0\}$
  such that all the isotropy groups
  of $\GC$ at points in $\Offen$ are conjugate.
\end{corollary}

\begin{corollary}\label{aut-if-n-tr}
  Let $p,q\in\Octc\setminus\C$.
  Then the following properties are
  equivalent:
  \begin{enumerate}
  \item
    $\Tr(p)=\Tr(q)$ and $\Nm(p)=\Nm(q)$.
  \item
    There is an automorphism $\phi\in \Aut(\Octc)$ such that
    $\phi(p)=q$.
  \end{enumerate}
\end{corollary}

Here we need a proof, since in \cite{BWH} we covered only
the quaternionic case.

\begin{proof}
  By construction, we have $\Nm(x)=x(\bar x)=B(x,x)$ for all $x\in\Octc$
  and $x\mapsto \frac 12\Tr(x)$ equals the projection of $x$ to $\C$
  with respect to the d
  irect sum decomposition $\Octc=\C\oplus W_{\C}$.

  Let $p=p'+p''$, $q=q'+q''$ with $p',q'\in\C$, $p'',q''\in W_{\C}$.
  Since $\Aut(\Octc)$ acts trivially on $\C$ and by linear,
  $B$-preserving transformations on $W_{\C}$, $(ii)$ implies that
  $p'=q'$ and $B(p'',p'')=B(q'',q'')$ which in turn implies
  $\Tr(p)=\Tr(q)$, $\Nm(p)=\Nm(q)$.

  Conversely,
  \[
  (  \Tr(p)=\Tr(q)) \land ( \Nm(p)=\Nm(q)))
  \implies
  (  p'=q') \land\left( B(p'',p'')=B(q'',q'')\right)
  \]
  and the latter condition implies
  $\exists\phi\in \Aut(\Octc):\phi(p)=q$ due to
  Proposition~\ref{VC-orbits}.
\end{proof}

\section{Automorphisms of $\Oct\tensor_\R\C$: Isotropy Groups}

  As a $\R$-vector space, $\Oct$ admits a direct sum decomposition
  $\Oct=C\oplus V$, where $C$ is the (real one-dimensional center),
  and $V=\{q\in\Oct: Tr(q)=0\}$.
  The euclidean scalar product on $V$ resp.~$\Oct$ extends to
  a complex bilinear form on $\WC$ resp.~$\Oct\tensor_\R\C$ which we
  denote by $B(\ ,\ )$.

  In the sequel, let $G$ denote the automorphism group
  of the $\R$-algebra $\Oct$ of octonions and let $\Gc$
  denote the automorphism group of the $\C$-algebra $\Octc$.

  \begin{proposition}\label{gen-SL3}
    For every $x\in \WC$ with $B(x,x)\ne 0$
    the isotropy group
    $I$ of $\GC$ at $x$ is isomorphic to $SL_3(\C)$.
  \end{proposition}

  \begin{figure}  
\begin{tikzpicture}[scale=2]
  \fill [red] (0,1.732) circle [radius=0.073] ;
\fill [red] (-1.5,0.866) circle [radius=0.073] ;
\fill [red] (1.5,0.866) circle [radius=0.073] ;
\fill [red] (-1.5,-0.866) circle [radius=0.073] ;
\fill [red] (1.5,-0.866) circle [radius=0.073] ;
\fill [red] (0,-1.732) circle [radius=0.073] ;
\draw (0,0) -- (-0.5,0.866);
\fill (-0.5,0.866) circle [radius=0.03] ;
\draw (0,0) -- (0.5,0.866);
\fill (0.5,0.866) circle [radius=0.03] ;
\draw (0,0) -- (-1,0);
\fill (-1,0) circle [radius=0.03] ;
\draw (0,0) -- (1,0);
\fill (1,0) circle [radius=0.03] ;
\draw (0,0) -- (-0.5,-0.866);
\fill (-0.5,-0.866) circle [radius=0.03] ;
\draw (0,0) -- (0.5,-0.866);
\fill (0.5,-0.866) circle [radius=0.03] ;
\draw (0,0) -- (0,1.732);
\fill (0,1.732) circle [radius=0.02] ;
\draw (0,0) -- (-1.5,0.866);
\fill (-1.5,0.866) circle [radius=0.02] ;
\draw (0,0) -- (1.5,0.866);
\fill (1.5,0.866) circle [radius=0.02] ;
\draw (0,0) -- (-1.5,-0.866);
\fill (-1.5,-0.866) circle [radius=0.02] ;
\draw (0,0) -- (1.5,-0.866);
\fill (1.5,-0.866) circle [radius=0.02] ;
\draw (0,0) -- (0,-1.732);
\fill (0,-1.732) circle [radius=0.02] ;
\end{tikzpicture}
\caption{Root system of $G_2$. {\color{red}Red: }
  Subsystem of type $A_2$ ($SL_3$) }
\end{figure}
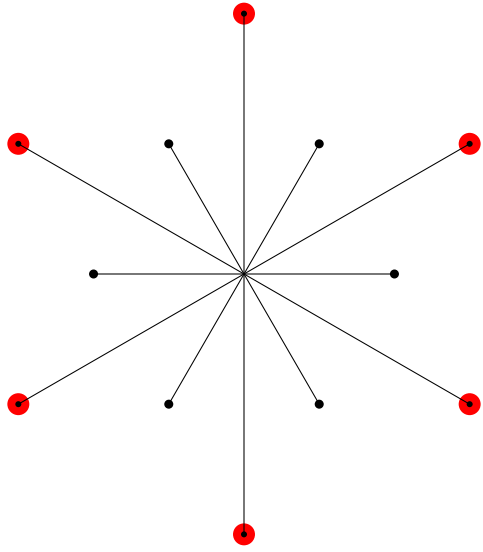

  \begin{proof}
    Let $c=B(x,x)$. Choose $\lambda\in\C$ and $y\in V\setminus\{0\}$
    such that
    $c=\lambda^2||y||^2$.
    Let $\lambda=\alpha+i\beta$, $\alpha,\beta\in \R$.
    Define $q=\lambda y$. Then $B(q,q)=B(x,x)$. Therefore
    $x$ and $q$ lie in the same $\GC$-orbit and consequently the
    isotropy groups at $q$ and $x$ are conjugate.

    We consider the isotropy group of the $G$-action at $q$.
    Since $q=\alpha y+ i\beta y$, the real part $\alpha y$
    and the imaginary part $\beta y$ of
    $q$ are $\R$-linearly dependent elements of $\Oct$.
    It follows that the isotropy of $G$
    at $q$ agrees with the isotropy of $G$ at $y$.

    The isotropy of $G$ at $y$ is isomorphic to $SU_3(\C)$
    (Corollary~\ref{oct-iso}).
    The isotropy $I$ of $\GC$ at $y$ is a complex Lie subgroup of $\GC$
    containing  the isotropy of $G$ at $q$. 

    The $\GC$-orbit through $q$ is a smooth affine quadric
    and therefore simply-connected.
    The fibration $\GC\to\GC/I$ induces
      the long exact homotopy sequence
     \[
     \ldots\to\underbrace{\pi_1(\GC/I)}_{=1}\to
     \pi_0(I)\to \underbrace{\pi_0(\GC)}_{=1}\to\ldots
     \]
    which yields that $I$ is connected.

    We have $\dimc(I)=8$, because the $\GC$-orbit is $6$-dimensional
    and $\dimc(\GC)=14$.

    Thus $\dimc(I)=\dimr(G\cap I)$ and $G\cap I$ is a
    {\em ``real form''} of $I$. Since $G\cap I\simeq SU_3(\C)$
    which is compact,
    $G\cap I$ is actually a maximal compact subgroup of $I$.    
    Every connected Lie group is homotopic to its maximal
    compact subgroup.
    Since $SU_3(\C)$ is simply-connected,
      it follows that $I$ is simply-connected.
      Hence $I$ is a simply-connected complex Lie group
      with a real form isomorphic
      to $SU_3(\C)$.
    This implies 
    $I\simeq SL_3(\C)$.
    
    \end{proof}
  \begin{remark}
    The figure shows in terms of the root system how $SL_3(\C)$
    occurs as a subgroup
    of the simple complex Lie group of type $G_2$.
    \footnote{For general information on semisimple Lie groups
      and root systems see e.g. \cite{MR1920389} or \cite{Humph}}
  \end{remark}

  \begin{corollary}\label{cor-iso}
    Let $U$ denote the
    Zariski open subset of $\OC=\C\oplus W_\C$
    defined as
    \[
    U=\{ (t,q)\in\C\oplus W_\C: B(q,q)\ne 0\}.
    \]
    
    For every $x,y\in U$ the isotropy groups of $\GC$ at $x$ and $y$
    are conjugate and isomorphic to $SL_3(\C)$.
  \end{corollary}

  \begin{proof}
    Follows from the above Proposition~\ref{gen-SL3}
    in combination with
    Corollary~\ref{omegaconj}.
  \end{proof}
  
  \begin{proposition}\label{spez-iso}
    For every $q\in \left(\WC\right)\setminus\{0\}$ with $B(q,q)= 0$
    the isotropy group
    $I$ of $\GC$ at $q$ is
    a semidirect product $I=S\ltimes U$
    where $S\simeq SL_2(\C)$ and
    $U$ is a $5$-dimensional unipotent normal subgroup of $I$.
  \end{proposition}

  \begin{proof}
        Let $q=x + iy\in(\WC)\setminus\{0\}$ with $x,y\in W$
    and $B(q,q)=0$. Then
    \begin{align*}
      &0=B(q,q)=B(x+iy,x+iy)= B(x,x)-B(y,y)+2iB(x,y)\\
      \implies\ 
      &B(x,x)=B(y,y),\quad B(x,y)=0.\\
    \end{align*}

    Since $B$ is the standard scalar product
    on $W\simeq\R^7$, $x$ and $y$ are orthogonal in $W$
    (because $B(x,y)=0$) and of the
    same length $r>0$ (because $B(x,x)=B(y,y)$).

    Define $x'=\frac xr$, $y'=\frac yr$. Now $x'$ and $y'$ are orthonormal.
    The set of all $z$ for which $(x',y',z)$ is a basic triple
    (as defined in Definition~\ref{def-basic-triple}) is the set
    of all elements $z$ of length $1$ in the orthogonal complement
    of $P=\left<x',y',x' \cdot y'\right>$.
    Since the automorphisms group of $\Oct$
    acts simply transitively on the set of basic triples, it follows
    that there is a $1\!:\!1$-correspondence between the unit sphere in
    $P^ \perp$
    and the set of all automorphisms which preserve both $x'$ and $y'$.
    Now a linear endomorphism preserves $x'$ iff it preserves $x$ and similarily
    for $y$ and $y'$. Hence the set of automorphisms preserving both $x'$ and
    $y'$ is simply the isotropy group of $G$ at $q=x+iy$.
    
    It follows that the isotropy group of $G$ at $q$ is isomorphic
    to $SU_2(\C)$ (which is the only real Lie group homeomorphic to
    the $3$-sphere).

    We recall that $G$ is a maximal compact subgroup of $\GC$.
    Let $I$ denote the isotropy group of $\GC$ at $q$ and
    let $K$ be a maximal compact subgroup of $I$.
    Now $K$ is a compact subgroup of $\GC$ and
      therefore conjugate to a subgroup of any maximal
      compact subgroup of $\GC$
      (\cite{Borel1998}, VII, Theorem 1.2(i))
      Thus
    \[
    \exists g\in\GC: gKg^{-1}\subset G.
    \]
    Observe that
    \[
       gKg^{-1}\subset gIg^{-1}=\{h\in\GC:h(g(q))=g(q)\}.
    \]
    Now
    \[
    \{z\in\WC:B(z,z)=0,\ z\ne 0 \}
    \]
    is one $\GC$-orbit. Hence the above considerations
    (showing that the $G$-isotropy at $q$ is isomorphic to
    $SU_2(\C)$) apply likewise to $g(q)$. 
    Thus
    \[
    gKg^{-1}\subset gIg^{-1}\cap G\simeq SU_2(\C)
    \]
    which in turn implies
    \[
    K \, \subset\, I\cap g^{-1}Gg\, \simeq\, SU_2(\C).
    \]
    Recall that $K$ is {\em maximal compact} in $I$
    and that $SU_2(\C)$ is compact.
    Hence
    \[
    K = I\cap g^{-1}Gg \,\simeq\, SU_2(\C).
    \]
    By standard Lie theory we have a {\em Levi decomposition}
    $I=S\ltimes U$  where $U$ is unipotent and $S$ is reductive.
    (\cite{MR0092928}, see also \cite{MR1064110}, Chapter~6.)
    Let
    $K$ be a maximal compact subgroup of $S$.
    From $K\simeq SU_2(\C)$ it follows that $S\simeq SL_2(\C)$.
      \end{proof}

\section{On the structure group of a certain  bundle}

We need a criterion that certain holomorphic maps are locally
trivial fiber bundles with a Lie group as structure group.

\begin{proposition}\label{loc-triv}
    Let $G$ be a complex Lie group acting holomorphically on a
    complex manifold $X$. Assume that all isotropy groups
    are connected and conjugate to a Lie subgroup $I$ of $G$.
    Let $D$ be a complex manifold.
    Let $H,F:D\to X$ be holomorphic maps and let
    \[
    V=\{(g,z)\in G\times D: g\cdot H(z)=F(z)\}
    \]

    Let $\pi:V\to D$ be the natural projection map:
    $\pi(g,z)=z$.
    Assume that $\pi$ is surjective.

    Then $\pi$ is a holomorphically locally trivial fiber bundle with
    structure group
    \[
    (N_G(I)/Z_G(I))\ltimes I
    \]
    and  fiber $I$.
    Here we use the notations
    \begin{align*}
      N_G(I) &= \{g\in G: \forall h\in I:
      ghg^{-1}\in I\}\quad & \text{(Normalizer)}\\
      Z_G(I) &= \{g\in G: \forall h\in I: ghg^{-1}=h \}
        \quad & \text{(Centralizer)}\\
    \end{align*}
  \end{proposition}

\begin{proof}
    Let $k=\dim(I)$.
    Let $M$ be the Grassmannian manifold pa\-ra\-me\-tri\-zing $k$-dimensional
    vector subspaces of the Lie algebra $Lie(G)$. The adjoint
    action of $G$ on $\Lie(G)$ naturally
    induces a $G$-action on $M$.
    Recall that by assumption all the isotropy groups
    \[
    I_z=\{g\in G: g(z)=z\}\quad (z\in X)
    \]
    are conjugate.
    Hence there is {\em one} $G$-orbit $Y$ in $M$
    which for {\em every} $z\in D$ contains
    the point of $M$ corresponding to the
    Lie algebra $W_z$ of the isotropy group at $H(z)$.
    We apply Proposition~\ref{loc-eq} to $G$ acting on $Y$
    and deduce that locally
    (near $p\in U\subset D$) on $D$ there are maps
    $\zeta:U\to G$ such that
    \[
    Ad(\zeta(z))\left( W_p \right) = W_z.
    \]
    Since the isotropy groups
    \[
    I_z=\{ g: g H(z)=H(z)\}\simeq I
    \]
    are connected, we deduce
    \[
    \zeta(z)I_p\zeta(z)^{-1}=I_z.
    \]
    Let $\psi:U\to G$ be a holomorphic map such that
    $\sigma:z\mapsto (\psi(z),z)$ is a local section
    $\sigma:U\to\pi^{-1}(U)$.
    
    We define a local trivialization
\[
\begin{matrix}
  I_p\times U & \simeq & \pi^{-1}(U) & \subset & V \\
  \downarrow && \downarrow && \downarrow \\
  U & = & U & \subset & D \\
\end{matrix}
\]
    via
    \[
    \pi^{-1}(U)\ni    (g,z)\stackrel{\Phi}\longrightarrow
    \left( \left( \zeta(z)\right)^{-1}\left(\psi(z)\right)^{-1}g\zeta(z)
    ; z \right) \in I_p\times U
    \]
    Let us check that $\Phi(g,z)\in I_p\times U$ for $(g,z)\in\pi^{-1}(U)$.
    First we observe that $(\psi(z),z),(g,z)\in V$ implies
    \[
    \psi(z)\cdot H(z)=F(z)= g\cdot H(z)
    \]
    which in turn implies
    \[
    \left(\psi(z)\right)^{-1}\cdot g\cdot H(z)=H(z)\iff
    \left(\psi(z)\right)^{-1}\cdot g\in I_z=      \zeta(z)I_p\zeta(z)^{-1}
    \]
    Hence
    \[
     \left( \zeta(z)\right)^{-1}\left(\psi(z)\right)^{-1}g\zeta(z)\in I_p
    \]
    
    Thus $\pi$ is holomorphically locally trivial.
  \end{proof}

\section{Reduction to the case $\Tr=0$}

\begin{lemma}\label{traceless}
  
  Let $ \Octc$ be the complexified algebra of
  octonions.
  Let $\GC=\Aut(\Octc)$.

  Let $\dom$ be a domain in $\C$.
  Let $F,H:\dom\to \Octc$ be holomorphic maps.
  Assume that $\Nm(F)=\Nm(H)$ and $\Tr(F)=\Tr(H)$
  (with $\Tr(F)=F+F^ c$ and $\Nm(F)=FF^ c$).

  Define $\hat F=\frac 12(F-F^c)$ and
  $\hat H=\frac 12(H-H^c)$.

  Then:
  \begin{enumerate}
  \item
    $\Tr(\hat F)=0=\Tr(\hat H)$.
  \item
    $\Nm(\hat F)=\Nm(\hat H)$.
  \item
    There exists a holomorphic map $\phi:\dom\to \GC$
    with $F(z)=\phi(z)\cdot(H(z))$ if and only
    there exists a holomorphic map $\phi:\dom\to \GC$
    with $\hat F(z)=\phi(z)\cdot(\hat H(z))$.
  \end{enumerate}
\end{lemma}

\begin{proof}
See \cite{BWH}, Lemma~14.1.
\end{proof}

\section{Vanishing orders}

Here we introduce the notion of {``\rm central divisors''}.
  As a preparation for this, we first discuss divisors for vector valued
  function.

\subsection{General maps}

Normally, {\em divisors} are defined for holomorphic 
functions with values in $\C$. Here we extend this notion to holomorphic
maps from Riemann surfaces to higher-dimensional complex vector spaces.

\begin{definition}\label{div-vc}
  Let $F:X\to V=\C^n$ be a holomorphic map from a Riemann surface $X$
  to a complex vector space $V=\C^n$. Assume $F\not\equiv 0$.

  The {\em divisor} of $F$ is the divisor corresponding to the
  pull back of the ideal sheaf of the origin, i.e., for $F=(F_1,\ldots,F_n)$,
  $F_i:X\to\C$ we have $div(F)=\sum_{p\in X} m_p\{p\}$
  where $m_p$ denotes the minimum of the multiplicities $mult_p(F_i)$.
\end{definition}
  
\begin{proposition}\label{prop-reduce}
Let $X$ be a non-compact Riemann surface, $V$ a complex vector space
and
$F,H:X\to V$ holomorphic maps
which are not identically zero.
Assume that $F,H$ have the same zero divisor.

Then there exists a holomorphic function
$\lambda:X\to\C$ and holomorphic maps $\tilde F, \tilde H: X\to V\setminus\{0\}$
such that $F=\lambda\tilde F$, $H=\lambda\tilde H$.
\end{proposition}

\begin{proof}
  Recall that on a non-compact Riemann surface every divisor
  is a {\em principal} divisor, i.e., the divisor of a holomorphic
  function.

  We choose a holomorphic function $\lambda$ on $X$ with
  \[
  div(\lambda)=div(H)=div(F)
  \]
  and define $\tilde F=F/\lambda$, $\tilde H=H/\lambda$.
\end{proof}

\begin{lemma}\label{div-unchanged}
Let $X$ be a Riemann surface, $V$ a vector space
and
$F:X\to V$, $\phi:X\to GL(V)$ be holomorphic maps, $F\not\equiv 0$.

Define $H(z)=\phi(z)\left(F(z)\right)$.

Then $F$ and $H$ have the same divisor.
\end{lemma}

\begin{proof}
  Let $div(F)=\sum_p m_p\{p\}$. Then for every $p\in X$ and
  $i\in\{1,\ldots,n\}$
  the germ of $F_i$ at $p$ is a divisible by $z_p^{m_p}$ where
  $z_p$ is a local coordinate with $z_p(p)=0$.
  Since $\phi(p)$ is linear, the components $H_i$ likewise have
  germs at $p$ which are 
  divisible by $z_p^{m_p}$.
  Hence $div(H)\ge div(F)$.

  The same arguments show that also $div(F)\ge div(H)$,
  since
  \[
  F(z)=\tilde\phi(z)\left(H(z)\right)
  \]
  for
  \[
  \tilde\phi(z)=\left(\phi(z)\right)^{-1}.
  \]

  Thus $div(F)=div(H)$.
\end{proof}

\subsection{Central divisors}\label{ss-cdiv}

\newcommand{\Cc}{C_{\C}}

In \cite{BW1}, Definition~3.1, we introduced the notion of a
{\em slice divisor}.
Here we will need a different notion of divisors.

Namely, we need a notion of divisor which measures
 where a given stem function assumes a value in
  the center $C_\C$ of $\Octc$.
This we call {``\em central divisor''}.

\begin{definition}\label{def-cdiv}
  Let $\Octc=\Oct\tensor_R\C$ be the complexification
  of the octonions with center
  $\Cc=\R\tensor_R\C\simeq\C$.
  $\dom\subset\C$ a domain, $F:\dom\to\Octc$ a holomorphic
  map. Assume $F(\dom)\not\subset \Cc$.

  The {\em central divisor} $\cdiv(F)$ is defined as the
  divisor (in the sense of Definition~\ref{div-vc})
  of the map from $\dom$ to $\Octc/\Cc$.
\end{definition}

Let $W$ denote the space of imaginary octonions, i.e.,
  \[
  W=\{q\in\Oct: \Tr(q)=0\}
  \]
Then 
$\Octc=\Cc\oplus\left(\WC\right)$ and
we can decompose $F:\dom \to\Octc$ as
\begin{equation}\label{F''}
  F=(F',F''):\dom\to \Cc\times(\WC)
\end{equation}
and the central divisor $\cdiv(F)$ equals
$\sum_{p\in\dom} n_p\{p\}$ where $n_p$ denotes the vanishing order of $F''$
at $p$.

\begin{example}
  Consider
  \[
  F(z)=1\tensor z + i\tensor z^2(z-1) + j\tensor z^3(z-1)^2.
  \]

  Then
  \[
  \cdiv(F)=
  2\cdot\{0\}+1\cdot \{1\}.
  \]
\end{example}

{\em Caveat:} These central divisors do {\em not} satisfy the
usual functoriality:
\begin{example}
  Let
  \[
  F(z)=1+i\tensor z,\quad H(z)=1+j\tensor (1+z)
  \]
  Then $\cdiv(F)=1\cdot\{0\}$ and $\cdiv(H)=1\cdot\{-1\}$, but
  \[
  \cdiv(FH)=\cdiv(1+i\tensor z+j\tensor (z+1)+k\tensor(z^2+z)
  \]
  is empty.
  Thus
  \[
  \cdiv(FH)\ne\cdiv(F)+\cdiv(H).
  \]
  (This is an example for $\H$, first presented in \cite{BWH}.
  But of course, $\H$ is a subalgebra of $\Oct$, so it is an
  example for the octonions as well.)
\end{example}

  \subsection{Central divisor
    for slice functions}\label{cdiv-stem}

  Let $f$ be a
  not slice-preserving slice regular function
  with associated stem function $F$.
Then we may simply define $\cdiv(f)$ as
\[
\cdiv(f)\stackrel{def}{=}\cdiv(F)
\]

Note that for a slice regular function $f$ on an axially
  symmetric domain $\Omega_D$ its central divisor $\cdiv(f)$ is
  a divisor on $D$ (and not on $\Omega_D$).
  
\section{Proof of the Main Theorem}\label{pf-main}

  We are now in a position to prove our main theorem
\ref{mainth}.

First, we consider slice preserving functions
(Lemma~\ref{main-slice-pre}).

Second, we deal with the
case where the image of $F$ is contained in the
null cone of the bilinear form $B$
(Proposition~\ref{F-in-null-cone}).

Third, we prove Proposition~\ref{F-not-in-null-cone}
which embodies the most difficult part of Theorem~\ref{mainth}.

Finally we complete the proof of Theorem~\ref{mainth}.

\begin{lemma}\label{main-slice-pre}
  Let $\Oct$ be the algebra of octonions,
  $\Octc=\Oct\tensor_\R\C$,
  $\GC=Aut(\Octc)$. Let $\dom\subset\C$ be a symmetric domain
  and let $\Omega_D\subset \Oct$ denote the corresponding axially symmetric
  domain.

  Let $f,h:\Omega_D\to \Oct$ be slice regular functions
  and let $F,H:\dom\to\Octc$
  denote the corresponding stem functions.

  Assume that $f$ is slice-preserving.
    Then the following are equivalent:
  \begin{enumerate}[label=(\roman*)]
  \item
    $f=h$.
  \item
    $F=H$.
  \item
    For every $z\in\dom$ there exists an element
    $\alpha\in Aut(\Octc)=\GC$ such that $F(z)=\alpha(H(z))$.
\item
    There is a holomorphic map $\phi:\dom\to\GC$
    such that $F(z)=\phi(z)\left(H(z)\right)\ \forall z\in\dom$.
  \end{enumerate}
\end{lemma}
\begin{proof}
  $(i)\iff(ii)\implies(iv)\implies(iii)$ is obvious.
  
  $f$ being slice preserving is equivalent to
  \[
  \forall z\in D: F(z)\in C_\C=\{q\in\Octc:q=\bar q\}\simeq\C
  \]
  (see Proposition~\ref{eq-s-p}).
  Now we obtain $(iii)\implies(ii)$, because the automorphism group
  $\Aut(\Octc)$ acts trivially on $C_\C$.
\end{proof}

\begin{proposition}\label{F-in-null-cone}
  Let $\Oct$ be the algebra of octonions
  and $\Oct_\C=\Oct\tensor_\R\C$.

  Let $\dom\subset\C$ be a symmetric domain.
  
  Let $F,H:\dom\to \Oct_\C\setminus\{0\}$ be holomorphic maps.

  Assume that
  $\Tr(F)=\Tr(H)=0$, $\Nm(F)=\Nm(H)$ 
  and that
  \[
  F(\dom)\subset\{v\in W_\C: B(v,v)=0\}.
  \]

  Then there exists a holomorphic map $\phi:\dom\to\GC$ such that
  \[
  \phi(z)\left(F(z)\right)=H(z)\ \forall z\in\dom.
  \]
\end{proposition}

\begin{proof}
  Due to our assumption
  $F(\dom)\subset\{v\in W_\C: B(v,v)=0\}$
  we know that $F(\dom)\subset W_\C$ and $\Nm(F)\equiv 0$.
  Thus $\Nm(H)=\Nm(F)\equiv 0$ which
  (in combination with $\Tr(H)=0$) implies
  \[
  H(\dom)\subset\{v\in W_\C: B(v,v)=0\}.
  \]
  Recall that
  $H_0=\{v\in W_{\C}:B(v,v)=0\}\setminus\{0\}$ is one
  orbit of $\GC=\Aut(\Octc)$ (Proposition~\ref{VC-orbits}).
  Hence $H_0\simeq\GC/I$ for some complex Lie subgroup $I$.
  Thus we may regard $F,H$ as holomorphic maps from $\dom$ to
  the quotient manifold $\GC/I$.
  Due to Proposition~\ref{spez-iso} we know that $I$ is connected.
  Now $\GC\to\GC/I$ is a $I$-principal bundle by standard Lie theory.
  If we pull-back this $I$-principal bundle via $F$ or $H$, we obtain
  an $I$-principal bundle over $\dom$ which admits a holomorphic section
  due to Proposition~\ref{section-2-con}.
  These sections induce liftings of the maps $F,H:\dom\to \GC/I$ to maps
  $\tilde F,\tilde H:\dom\to \GC$.
  \begin{center}
    \newcommand{\komma}{,} %
    \begin{tikzcd}[scale=4]
      & \GC \arrow[d,""]
      \\
      D \arrow[r,"F\komma H"']  \arrow[ru,"\tilde F\komma\tilde H"] & \GC/I
      \\
    \end{tikzcd}
  \end{center}

  Now we may define
  the desired map $\phi$ using the group structure of $\GC$
  \[
  \phi(z)=\tilde H(z)\cdot\left(\tilde F(z)\right)^{-1}.
  \]
  Then
  \[
  \phi(z) \cdot\tilde F(z)=\tilde H(z)\quad\implies\quad
  \phi(z) \left(F(z)\right)=H(z).
  \]
\end{proof}

\begin{proposition}\label{F-not-in-null-cone}
  Let $\Oct$ be the algebra of octonions
  and $\Oct_\C=\Oct\tensor_\R\C$.

  Let $\dom\subset\C$ be a symmetric domain.
    
  Let $F,H:\dom\to \Oct_\C\setminus\{0\}$ be holomorphic maps
  such that $\overline{F(\bar z)}=F(z)$,
  $\overline{H(\bar z)}=H(z)$,

    Assume that
      $\Tr(F)=\Tr(H)=0$, $\Nm(F)=\Nm(H)$ 
     and that
    \[
    F(\dom)\not\subset \{q\in\Octc:B(q,q)=0\}.
    \]
  
  Then there exists a holomorphic map $\phi:\dom\to\GC$ such that
  \[
  \phi(z)\left(F(z)\right)=H(z)\ \forall z\in\dom.
  \]
\end{proposition}

\begin{proof}
  Throughout the proof we will use
    the fact that $\dom$ is a non-compact Riemann surface.

    Note that we assume $\Tr(F)=\Tr(H)=0$.
    It follows that the images $F(\dom),H(\dom)$ are contained in $\WC$.

    Thus we may regard $F$ and $H$ as holomorphic maps from $\dom$ to
    $\left(\WC\right)\setminus\{0\}$.

From our assumption on $F$ and $H$ we deduce that for every $z\in\dom$ there
is an element $g\in G_\C$ with $H(z)=gF(z)$ (Corollary~\ref{aut-if-n-tr}).

The case where both $F$ and $H$ are constant is trivial.
Hence we may assume that at least one of the two maps is not constant.
Without loss of generality, we assume $F$ to be non-constant.

We define a complex space $V$ and a projection map $\pi:V\to\dom$:
\begin{align*}
  &V=\{(\alpha,z)\in\GC\times\dom : F(z)= \alpha H(z)\}\\
  &\pi:(\alpha,z)\mapsto z.\\
\end{align*}

The isotropy groups for the
$\GC$-action on $\WC\setminus\{0\}$ have all the same
dimension (namely $8$)
due to Corollary~\ref{dim-c-isotropy}. 

Therefore we may apply Proposition~\ref{loc-eq}.
It follows that for every $p\in\dom$ there is an open neighborhood $M$
of $p$ in $\dom$
and a holomorphic map $\psi:M\to \GC$ with
\[
F(z)=\psi(z)H(z)\ \,\, \forall z\in M.
\]
In other words:
There are everywhere local holomorphic sections for $\pi:V\to \dom$.

Define
\begin{equation}
  \label{D-def}
L=\{z\in\dom:B(F(z),F(z))=0\},\ \Offen=\dom\setminus L
\end{equation}

Recall that %
\[
F(\dom)\not\subset\{v: B(v,v)=0\}.
\]

Thus $\Offen$ as defined in \eqref{D-def}
above is an open dense subset of $\dom$.

We recall that for $x\in\Offen$ the isotropy group at $x$
is isomorphic to $SL_3(\C)$
(Proposition~\ref{gen-SL3})
and therefore in particular simply-connected.%
\footnote{In the case $A=\H$ this group is isomorphic to $\C^*$ and
  thus not simply-connected. One reason, why we need different proofs
  in the two different cases.}

Recall
\[
V=\{(\alpha,z)\in\GC\times\dom : F(z)= \alpha H(z)\}
\]
Let $\pi:V\to\dom$  denote the natural projection
$(\alpha,z)\mapsto z$.
For $z\in\Offen$ we have $\pi^ {-1}(z)\simeq SL_3(\C)$.
Fix a point $p\in \Offen$ and let
\[
I=\{\alpha\in \GC:\alpha\cdot F(p)=F(p) \}.
\]

From Proposition~\ref{loc-triv} we deduce that $V\to \dom$ restricts
to a holomorphically locally trivial fiber bundle on $\Offen$
with 
fiber $I\simeq SL_3(\C)$ and a structure group which
is isomorphic to
\[
\left( N_{G_\C}(I)/Z_{G_\C}(I) \right)  \ltimes I
\]

Thus $\pi:V\to  \dom$ restricts to a locally trivial fiber bundle
over $\Offen$ whose structure group is a
(not necessarily connected) complex Lie group.

In view of the fact that $SL_3(\C)$ is simply-connected, we  infer
from Proposition~\ref{global-con} that $\pi:V\to\dom$ admits
a global continuous section.

Finally, we obtain a global holomorphic section (equivalently: a holomorphic
map $\phi$ with the desired properties)  by Proposition~\ref{oka}.
\end{proof}

\begin{proposition}\label{mainmainth}
  Let $\Oct$ be the algebra of octonions
  and $\Oct_\C=\Oct\tensor_\R\C$.

  Let $\dom\subset\C$ be a symmetric domain.
  
  Let $F,H:\dom\to \Oct_\C\setminus\{0\}$ be holomorphic maps.

  Assume that $\Tr(F)=\Tr(H)=0$ and $\Nm(F)=\Nm(H)$.
  
  Then there exists a holomorphic map $\phi:\dom\to\GC$ such that
  \[
  \phi(z)\left(F(z)\right)=H(z)\ \forall z\in\dom.
  \]
\end{proposition}

\begin{proof}
  Follows from Proposition~\ref{F-in-null-cone}
  and Proposition~\ref{F-not-in-null-cone}.
\end{proof}

\begin{proof}[Proof of the Theorem~\ref{mainth}]
    The assertions of part $b)$ have been proved in
  Lemma~\ref{main-slice-pre}.

  Thus we may assume without loss of generality that neither $f$
  nor $h$ is slice preserving.
    
  We proceed as follows:
\[
  \begin{tikzcd}
    (i) \arrow[r,Leftrightarrow]
    & (ii) \arrow[r,Leftrightarrow]
     \arrow[d,Rightarrow]
    & (iii)\\
     & (iv)
    \arrow[ur,Rightarrow]
      \\
  \end{tikzcd}
\]

  $(ii)\implies(iv)$:

  By assumption we have  $\Tr(F)=\Tr(H)$.
  Define
  \[
  \hat F=\frac 12\left(F-F^ c\right),\quad
  \hat H=\frac 12\left(H-H^ c\right)
  \]
  Evidently $\Tr(\hat H)=\Tr(\hat F)=0$.
  Moreover $\Nm(F)=\Nm(H)$ in combination with Lemma~\ref{traceless}
  implies that $\Nm(\hat F)=\Nm(\hat H)$.

  With respect to the decomposition $\Oct_\C=\C\oplus\left(\WC\right)$ the map
  $\hat F$ resp.~$\hat H$ is just the second component of
  $F$ resp.~$H$.
  Recall that we discuss the case where neither $f$ nor $h$ is
  slice preserving. Hence neither $\hat F$ nor $\hat H$ are
  vanishing identically.
  
  By the definition of the central divisor (introduced in
  \paragraph\ref{ss-cdiv}) we may conclude that
  $\cdiv(F)=\cdiv(\hat F)$ and $\cdiv(H)=\cdiv(\hat H)$.
  
  Since $\cdiv(F)=\cdiv(H)$, it follows that there are holomorphic maps
  $\tilde F,\tilde H:\dom\to\Octc\setminus\{0\}$ and
  $h:\dom\to\C$ such that
  \[
  \hat F=h\tilde F, \quad\hat H=h\tilde H
  \]
  (Proposition~\ref{prop-reduce}, here multiplication
  by $h(z)$ means multiplying elements of $\Octc$
  via $\C\simeq{1}\tensor\C\subset\Oct\tensor_\R\C$.)
  
  Observe that
  \begin{align*}
  &0=\Tr(\hat F)=hTr(\tilde F),\ \Nm(\hat F)=h^ 2N(\tilde F)\\
  &0=\Tr(\hat H)=hTr(\tilde H),\ \Nm(\hat H)=h^ 2N(\tilde H)\\
  \end{align*}
  Hence $\Tr(\tilde F)=0=\Tr(\tilde H)$ and $\Nm(\tilde F)=\Nm(\tilde H)$
  and Proposition~\ref{mainmainth} implies that there is a holomorphic map
  $\phi:\dom\to\GC$ such that
  \[
  \phi(z)\left(\tilde F(z)\right)=\tilde H(z)\ \forall z\in\dom.
  \]
  which in turn implies
  \[
  \phi(z)\left(\hat F(z)\right)=\hat H(z)\ \forall z\in\dom,
  \]
  because $\Gc$ acts linearly,
  $\hat F=h\tilde F$ and $\hat H=h\tilde H$.
  Finally
  \[
  \phi(z)\left(F(z)\right)=H(z)\ \forall z\in\dom
  \]
  follows via Lemma~\ref{traceless}.
  
  $(iv)\implies(iii)$:
  The implication $(iv)\implies \cdiv(F)=\cdiv(H)$ is due to
  Lemma~\ref{div-unchanged}, the other assertion is obvious.
  
  For $(iii)\iff(ii)$ see Corollary~\ref{aut-if-n-tr}.

  For $(i)\iff(ii)$, see
  Proposition~\ref{n-tr-stem} and
  Section \ref{cdiv-stem}. 

  \end{proof}

\nocite{FH,SV,F2,B,G,MM,F,BW2,BW1,BW3,AB2,AB1,BS,BG,BM,GS2,GS1,GP,GPS}

\end{document}